\documentclass[11pt]{amsart}
\usepackage[utf8]{inputenc}

\usepackage{maths}
\usepackage[normalem]{ulem}
\usepackage{bm}
\usepackage{graphicx}
\usepackage{subcaption}
\usepackage{verbatim}
\usepackage[normalem]{ulem}
\usepackage{overpic} 
\usepackage{graphicx}
\usepackage{subcaption} 
\usepackage{booktabs}
\usepackage{siunitx}
\usepackage{xcolor}
\definecolor{darkgreen}{RGB}{0,100,0}
\definecolor{darkorange}{RGB}{200,100,0}
\definecolor{darkred}{RGB}{170,0,0}
\usepackage{pifont}
\usepackage{tikz}

\usepackage{fontawesome5}

\newcommand{\cmark}{\textcolor{darkgreen}{\ding{51}}}
\newcommand{\cmarkm}{\textcolor{darkorange}{(\ding{51})}}
\newcommand{\xmark}{\textcolor{darkred}{\ding{55}}}

\allowdisplaybreaks

\calclayout

\usepackage[normalem]{ulem}

\makeatletter

\let\oldmaketitle\maketitle

\renewcommand{\maketitle}{%
  \oldmaketitle
  \vspace{-2cm}
  \begingroup
    \parindent=0pt
    \begin{center}
     \@setaddresses
    \end{center}
  \endgroup
  \global\let\@setaddresses\relax
}

\makeatother

\title[Bregman Proximal Gradient Methods for KL Regression]{On The Linear Convergence of Bregman Proximal Gradient Methods with Applications to Kullback--Leibler regression}

\author[J.\ Chirinos-Rodriguez]{Jonathan\ Chirinos-Rodríguez${}^{1,*}$}
\thanks{${}^*$Corresponding author.}

\author[C.\ Daniele]{Christian\ Daniele${}^{2}$}
\author[C.\ Févotte]{Cédric\ Févotte${}^{1}$}
\author[E.\ Soubies]{Emmanuel\ Soubies${}^{1}$}

\address{${}^1$ University of Toulouse, Toulouse INP, CNRS, IRIT, France}
\email{\{jonathan-eduardo.chirinos-rodriguez, emmanuel.soubies, cedric.fevotte\}@irit.fr}

\address{${}^2$ MaLGa Center, DIBRIS, University of Genoa, Italy \newline 
Laboratoire J. A. Dieudonn\'e,
Universit\'e C\^ote d’Azur, France}
\email{christian.daniele@edu.unige.it}

\begin{document}
\sloppy
\maketitle

\begin{abstract}
Bregman Proximal Gradient methods (BPGM) exploit the underlying geometry of the objective function through a carefully chosen mirror map. In this work, we introduce a novel notion of strong convexity, termed Restricted Relative Strong Convexity, and establish linear convergence rates for BPGM under this condition. We then exploit the proposed theoretical framework to provide an in-depth analysis of the convergence of BPGM for (regularized) Kullback--Leibler regression problems, covering scenarios with both unique and non-unique minimizers, as well as regularized and unregularized formulations. Specifically, we demonstrate that using the popular Burg's entropy as a distance-generating function may only yield linear convergence for certain KL regression problems. In contrast, we show that employing a smoothed version of the Burg's entropy induces the suitable geometry required to guarantee linear convergence. We conclude with numerical experiments that nicely align with our theoretical findings.
\end{abstract}

\section{Introduction}
Many problems in fields such as machine learning, computational imaging, signal processing, and variational inverse problems can be formulated as the minimization of a convex composite functional. Specifically, these are problems of the form
\begin{equation}\label{eq:comp_obj}
    \argmin_{\bx \in \cC} J(\bx):=F(\bx) + R(\bx),
\end{equation}
where $F\colon\R^N\to(-\infty,\infty]$ is proper, convex and differentiable, and $R\colon\R^N\to(-\infty,\infty]$ is proper, convex and lower semicontinuous, but not necessarily differentiable. Moreover, $\cC\subseteq \mathrm{int}(\dom J)$ denotes a constraint set, which we assume to be closed and convex. Finally, we assume that the solution set, denoted by $\mathcal{S}$, is nonempty, and we denote $\bar J:=\inf_{\bx\in\cC}J(\bx)>-\infty$ the optimal value for Problem~\cref{eq:comp_obj}. \\

\paragraph{\textbf{Proximal Gradient Method (PGM)}} 
In order to find solutions of~\cref{eq:comp_obj}, a method of choice is the so-called PGM, also known as forward-backward splitting algorithm~\cite{CW2005,lions1979,Passty1979}. Given $\bx^0 \in \cC$, its iterates write: 
\begin{equation}\label{eq:PGM}
\bx^{k+1} = \argmin_{\bx \in \cC} F(\bx^k) + \left<\nabla F(\bx^{k}), \bx - \bx^{k} \right> + \frac1\tau \|\bx-\bx^k\|_2^2 + R(\bx),
\tag{PGM}
\end{equation}
where $\tau>0$ denotes a constant step-size. The convergence of~\cref{eq:PGM} is classically ensured by assuming Lipschitz smoothness of the gradient (or simply $L$-smoothness) of $F$. 
Under this assumption, and for suitable choices of the step-size $\tau>0$, sublinear convergence rates---of the objective function values and sometimes of the iterates---, i.e., of order $\mathcal{O}(1/k)$ are well-established~\cite{Nesterov2013}, with further accelerated versions providing rates of order $\mathcal{O}(1/k^2)$ studied in~\cite{Beck2009,Nesterov1983}.

It is also well-known that if the differentiable term $F$ is, in addition, strongly convex, then one can establish a linear rate of convergence, i.e., of the form $\mathcal{O}(\rho^k)$ for some $\rho\in(0,1)$, see e.g.~\cite{Bredies_2008,Nesterov2013}. We recall that $F$ is \emph{strongly convex} if there exists $\mu>0$ such that, for all $\bx, \bx' \in \mathrm{int}(\dom F)$
\begin{equation}\label{eq:strong_conv_intro}
\langle \nabla F(\bx)-\nabla F(\bx'), \bx-\bx'\rangle\geq\mu\|\bx-\bx'\|_2^2.
\end{equation}
 While $L$-smoothness is satisfied by a broad class of functions (possibly with a large, i.e., loose, constant $L$), strong convexity is much more restrictive~\cite{Li2017}. To illustrate this, consider the least-squares regression problem, where the smooth term equals $F=(1/2)\|\bA\cdot-\by\|_2^2$ for some linear map $\bA\in\R^{M\times N}$.
In this case, it is easy to see that $F$ is always $L$-smooth but, however, it is strongly convex only if $\bA$ is full column rank, a condition that strictly excludes many practical situations such as common underdetermined inverse problems. Beyond the least-squares setting, strong convexity fails for any function of the form $F = G(\bA \cdot)$ with $G\colon \mathbb{R}^M \to \mathbb{R}_{\geq 0}$ and $\bA$  not full column rank. Furthermore, even when $\bA$ has full column rank, strong convexity still fails for such $F$ when $G$ itself is not strongly convex on $\mathrm{Im}(\bA)$. This is typically the case, for instance, for Poisson regression~\cite{harmany2011spiral}, where $G$ coincides with the generalized Kullback-Leibler divergence.

However, in practice, linear convergence may still be observed in some cases where strong convexity fails to hold. This has led numerous studies to identify weaker sufficient conditions that ensure linear convergence of~\eqref{eq:PGM}~\cite{Attouch2013,beck2017,Drus2018,Luo1993,Neco2016,necoara2019linear,zhang2015,Zhou2017}. In particular, and for further purposes, we focus on two specific properties. The first one is the so-called Kurdyka--{\L}ojasiewicz, introduced in~\cite{Kurdyka1998,loj1963} (see also~\cite{Bolte2007}), and
linear convergence under this condition was shown in~\cite{Attouch2013,Peypo2017,Li2017} functions.
 A detailed comparison between the Kurdyka--{\L}ojasiewicz property and related conditions can be found in~\cite{karimi2016}. The second one is introduced by the authors in~\cite{Lai2013} (see also~\cite{liu2014,Gha2014,zhang2015}), who observed that \emph{global} strong convexity is not strictly necessary to prove linear convergence of~\eqref{eq:PGM}. Instead, it suffices to require strong convexity of $F$ only between a point and its projection onto the solution set $\mathcal{S}$. As defined in~\cite[Lemma 7]{Lai2013}, $F$ is said to be \emph{restricted strongly convex} (RSC) (or to have \emph{quadratic gradient growth}, see \cite[Definition 3]{necoara2019linear}) if there exists $\mu >0$ such that, for all $\bx\in\mathrm{int}(\dom F)$ 
\begin{equation}\label{eq:RSC}
\tag{RSC}
\langle \nabla F(\bx)-\nabla F(\bar\bx_{\ell_2}), \bx-\bar\bx_{\ell_2}\rangle\geq\mu\|\bx-\bar\bx_{\ell_2}\|_2^2,
\end{equation}
where $\bar\bx_{\ell_2}$ denotes the Euclidean projection of $\bx$ onto $\mathcal{S}$. In particular, if $\cS=\{\bar\bx\}$, then $\bar\bx_{\ell_2}=\bar\bx$. Actually, this condition enables the proof of linear convergence rates for several first order algorithms~\cite{Lai2013,necoara2019linear,Gha2014,zhang2015}. We particularly highlight the work of Necoara et al.~\cite{necoara2019linear}, which provides a thorough hierarchichal study between conditions depending on the solution set $\cS$ such as RSC, \emph{quasi-strong convexity} or \emph{quadratic functional growth}.\\

\paragraph{\textbf{Bregman Proximal Gradient Method (BPGM)}}
Recently, BPGM~\cite{bauschke2017descent} (see also~\cite{Bolte2018,lu2018relatively,Teboulle2018,Zhou2019})---or simply Bregman Gradient Method (BGM) when $R \equiv 0$~\cite{AusTeb2006,Tseng2010}---has been introduced. Then, BPGM is given by the following iteration: given $\bx^0\in \cC \cap \mathrm{int}(\dom\varphi)$,
\begin{equation}\label{eq:BPGM}
\tag{BPGM}
\bx^{k+1} = \argmin_{\bx \in \cC} F(\bx^k) + \left<\nabla F(\bx^{k}), \bx - \bx^{k} \right> + \frac1\tau D_{\varphi}(\bx,\bx^k) + R(\bx),
\end{equation}
where $D_\varphi$ denotes the  Bregman divergence relative to $\varphi$. If $\varphi=(1/2)\|\cdot\|_2^2$, then standard PGM is recovered. Otherwise, the choice of the generating function $\varphi$ allows the algorithm to better adapt to the geometry of the cost function $J$. For instance, the gradient of the smooth term $F$ may not satisfy the standard Lipschitz continuity assumption, but rather the so-called \emph{Lipschitz like/Convexity condition} (or simply LC condition)~\cite{bauschke2017descent} with respect to the reference function $\varphi$ (see also~\cite{Bolte2018}, where such condition is termed \emph{L-smad property}, or~\cite{lu2018relatively}, where $F$ is said to be \emph{$L$-Lipschitz relative to $\varphi$}). We recall that $F$ satisfies the LC condition if there exists $L>0$ such that the function $LF-\varphi$ is convex or, equivalently, if 
\begin{equation}\label{eq:rel_Lip}
F(\bx)\leq F(\bx')+\langle\nabla F(\bx'),\bx-\bx' \rangle +L D_\varphi(\bx,\bx'),
\end{equation}
for all $\bx,\bx'\in\mathrm{int}(\dom F)\cap\mathrm{int}(\dom \varphi)$. The above condition clearly motivates the expression of~\cref{eq:BPGM}. Notable examples that satisfy the LC condition~\cref{eq:rel_Lip} and, however, do not have Lipschitz continuous gradient, can be found in~\cite[Section 2]{bauschke2017descent} and~\cite[Section 2]{lu2018relatively}. Under this LC condition, sublinear convergence rates for~\eqref{eq:BPGM} were provided in~\cite{bauschke2017descent} when $\tau < 1/L$. Moreover,  linear rates were established in~\cite{lu2018relatively,Teboulle2018} under a generalization of (global) strong convexity~\cref{eq:strong_conv_intro} to the Bregman setting. Specifically, we say that $F$ is \emph{strongly convex relative to $\varphi$} (see~\cite[Definition 4.1]{Bartlett2007}) if there exists $\mu>0$ such that
\begin{equation}\label{eq:breg_rsc}
\langle \nabla F(\bx)-\nabla F(\bx'), \bx-\bx'\rangle\geq\mu\langle \nabla \varphi(\bx)-\nabla \varphi(\bx'), \bx-\bx'\rangle,
\end{equation}
for all $\bx,\bx'\in\mathrm{int}(\dom F)\cap\mathrm{int}(\dom \varphi)$. Again, when $\varphi=(1/2)\|\cdot\|_2^2$ classical strong convexity is recovered. While the possibility to adapt $\varphi$ makes this condition less restrictive than strong convexity and thus extends the class of functions for which linear convergence can be proven (see~\cite[Section 2]{lu2018relatively}), it remains a rather restrictive condition in practice. Coming back to the case  $F = G(\bA \cdot)$, 
one can show
that if $\bA$ is not full column rank then $\mu = 0$ for any choice of $\varphi$. This leads to a natural question: is it possible to extend the aforementioned weaker alternatives to strong convexity (namely Kurdyka--{\L}ojasiewicz \ or~\cref{eq:RSC}) to the Bregman setting and prove linear convergence rates for~\cref{eq:BPGM}? While such an extension has been proposed for {\L}ojasiewicz-type conditions by the authors in~\cite{Bauschke2019} (allowing for the analysis of nonconvex functions too), a generalization of~\cref{eq:RSC} to the Bregman setting, along with an associated linear convergence proof for~\cref{eq:BPGM}, remains unexplored.\\

\paragraph{\textbf{Contributions and outline}}
First, we establish in~\Cref{sec:BPGM} sufficient conditions ensuring linear convergence rates for~\cref{eq:BPGM}, as well as for an automatic step-size selection rule via backtracking. Our result is built upon a novel property, termed \emph{Restricted Relative Strong Convexity}, that extends the Restricted Strong Convexity property~\cref{eq:RSC} to the Bregman setting. Next, in~\Cref{sec:KL} we exploit this theoretical framework to analyze the convergence of~\cref{eq:BPGM} for (regularized) Kullback--Leibler (KL) problems. Specifically, we specify in~\Cref{sec:KL_theory} suitable validity regions for linear convergence across three different distance generating functions: the squared $\ell_2$-norm, Burg's entropy, and our proposed smoothed Burg's entropy. Finally,~\Cref{sec:experiments} corroborates our theoretical findings with numerical experiments on synthetic data, including a comparison of~\cref{eq:BPGM} under the three different choices of $\varphi$ against the well-known Richardson--Lucy algorithm.\\

\paragraph{\textbf{Notation and definitions.}} 
In the following, $\R$ and $\R_{\geq 0}$ will denote the sets of real and non-negative real numbers, respectively. Given a subset $\mathcal{B}$ of $\R^N$, we denote by $\mathrm{int}(\cB)$ its interior, by $\partial\cB$ its boundary, by $\overline{\cB}$ its closure, and by $\cB^c$ its complementary; i.e., $\cB^c=\R^N\setminus\cB$. Moreover, given $\bx=(x_1,\ldots, x_N)\in\R^N$, $\sigma(\bx)$ denotes its support, i.e., $\sigma(\bx)=\{n \mid x_n\neq 0 \}$. For a given set of indices $\omega\subseteq\{1,\ldots, N\}$, we denote by $\#\omega$ its cardinality. Given a matrix $\bA\in\R^{M\times N}$, $\bA_\omega$ defines its restriction to the columns indexed by $\omega$: $\bA_\omega:=(\ba_{\omega[1]},\ldots,\ba_{\omega[\# \omega]})$, where $\ba_{\omega[n]}$, $n=1,\ldots,\# \omega$, denotes the $\omega[n]^{\mathrm{th}}$ column of $\bA$. Finally, we denote by $\angle(\bx,\mathcal{H})\in[0,\pi/2]$ the angle between $\bx$ and its orthorgonal projection onto the hyperplane $\mathcal{H}$.
In the context of convex analysis, we denote by $\Gamma_0(\R^N)$ the set of functions $\varphi\colon \R^N\to(-\infty,\infty]$ that are proper, convex and lower semicontinuous. Next, the \emph{Bregman divergence} relative to $\varphi\in\Gamma_0(\R^N)$ is defined, for all $\bx\in\dom \varphi$, $\bx'\in\mathrm{int}(\dom \varphi)$,~as
$$
D_\varphi(\bx,\bx'):= \varphi(\bx)-\varphi(\bx')-\langle s_\varphi(\bx'),\bx-\bx'\rangle
$$
and $+\infty$ elsewhere, for some $s_\varphi(\bx')\in\partial \varphi(\bx')$,  where $\partial \varphi(\bx)$ denotes the subdifferential of $\varphi$ at $\bx$ (\cite[Definition 16.1]{BCombettes}).
Note that we require  $\bx'$ to belong to $\mathrm{int}(\dom \varphi)$ since in this case the subdifferential  of $\varphi$ is nonempty (see~\cite[Theorem 9.23]{BCombettes}). If $\varphi$ is differentiable at $\bx$, then the subdifferential of $\varphi$ at $\bx$ is a singleton, and coincides with the gradient of $\varphi$ at $\bx$, i.e., $\partial \varphi(\bx)=\{\nabla \varphi(\bx)\}$. If $\varphi$ is not differentiable, the above definition depends on the choice of the specific element (subgradient) $s_\varphi(\bx')$ of the subdifferential $\partial \varphi(\bx')$. When such a choice is not specified, we choose the one with minimal Euclidean norm. 

We next recall the definition of a Legendre function $\varphi\in\Gamma_0(\R^N)$, see, e.g.,~\cite[Chapter 26]{Rock1970}: we say that $\varphi$ is \emph{essentially smooth} if it is differentiable on $\mathrm{int}(\dom \varphi)$ and $\lim_{n\to+\infty}\|\nabla \varphi(\bx_n)\|=+\infty$ for every sequence $(\bx_n)_{n\in\mathbb{N}}$ converging to a boundary point of $\mathrm{int}(\dom \varphi)$. Then, we say that $\varphi$ is a \emph{Legendre function} if 1) $\mathrm{int}(\dom \varphi)\neq\emptyset$, 2) it is essentially smooth, and 3) it is strictly convex on $\mathrm{int}(\dom \varphi)$.

Finally, given $\varphi\in\Gamma_0(\R^N)$ a Legendre function such that $\cS \cap \mathrm{int}(\dom(\varphi)) \neq \emptyset$, we define the \emph{Bregman projection} of $\bx\in \mathrm{int}(\dom(\varphi))$  onto $\cS$ by
\begin{equation}\label{eq:bregproj}
    \bar\bx_\varphi := \argmin_{\bz \in \mathcal{S} \cap \dom \varphi} D_\varphi(\bz,\bx).
\end{equation}
Moreover, since we assumed that $\varphi$ is a Legendre function, $\bar\bx_\varphi$ exists, is unique, and belongs to $\cS \cap \mathrm{int}(\dom(\varphi))$ see~\cite[Theorem 3.12]{Bauschke1997}. As a notable example, we denote by $\bar\bx_{\ell_2}$ the Euclidean projection of $\bx$ onto $\mathcal{S}$, namely $\bar\bx_\varphi$ with $\varphi=(1/2)\|\cdot\|^2_2$. Finally, if $\mathcal{S}=\{\bar\bx\}$, then indeed $\bar\bx_\varphi=\bar\bx$. We are now ready to present the main section of this work.

\section{The Bregman Proximal Gradient Method}\label{sec:BPGM}
As already mentioned, we are interested in this work in analyzing the convergence of Bregman (Proximal) Gradient Methods. In particular,  we aim at deriving novel sufficient conditions that ensure linear convergence rates and allow to cover a larger class of problems with respect to existing literature. To that end, we will first introduce some concepts that are specific to this work. First, and unlike~\cite{lu2018relatively,Teboulle2018}, which assume that $F$ is globally strongly convex relative to $\varphi$, our analysis relies only on the following less stringent notion, which extends the restricted strong convexity property~\cite{Lai2013,necoara2019linear,zhang2015} to our Bregman setting. 

\begin{definition}[Restricted Relative Strong Convexity]\label{def:RRSC} 
We say that $F$ is \emph{restricted strongly convex relative to $\varphi$} if there exists a constant $\mu >0$ such that, for all $\bx\in\mathrm{int}(\dom F)\cap\mathrm{int}(\dom \varphi)$, we have
\begin{equation}\label{eq:rrsc}
\tag{RRSC}
\langle \nabla F(\bx)-\nabla F(\bar\bx_{\ell_2}), \bx-\bar\bx_{\ell_2}\rangle\geq \mu \langle \nabla \varphi(\bx)-\nabla \varphi(\bar\bx_{\ell_2}), \bx-\bar\bx_{\ell_2}\rangle.
\end{equation}
\end{definition}

Clearly, when $\varphi=(1/2)\|\cdot\|_2^2$, we recover~\cref{eq:RSC}.
Second, we introduce a (restricted) symmetry measure for the Bregman divergence $D_{\varphi}$.

\begin{definition}[Restricted symmetry coefficient]\label{def:alpha(h)} Let $\cD\subseteq\mathrm{int}(\dom \varphi)$. Then, the (restricted) symmetry coefficient $\alpha_\cD(\varphi)$ is given by
\begin{equation}\label{eq:alpha(h)}
\alpha_\cD(\varphi):=\inf \left\{\frac{D_{\varphi}(\bx,\by)}{D_{\varphi}(\by, \bx)} \mid \bx \in \mathrm{int}(\dom \varphi), \by\in \cD, \bx\neq\by \right\}\in[0, 1].
\end{equation}
\end{definition}
If $\cD=\mathrm{int} (\dom \varphi)$ then we simply denote the above as $\alpha(\varphi)$, and we recover the global symmetry coefficient introduced in~\cite{bauschke2017descent}. In this case, we recall that $\alpha(\varphi)=1$ if and only if $D_{\varphi}$ is symmetric, and $D_{\varphi}$ is symmetric if and only if $\varphi$ is quadratic, see e.g.~\cite{bauschke2017descent}. On the other hand, $\alpha(\varphi)=0$ indicates a total absence of symmetry. This is the case, for instance, when $\varphi$ equals the Boltzman--Shannon or the Burg entropies~\cite{bauschke2017descent}.  Additionally, it is easy to observe that, for any $\cD\subset \mathrm{int}(\dom \varphi)$, it holds $\alpha_\cD(\varphi)\geq \alpha(\varphi)$ since, clearly, $\alpha(\varphi)$ is defined on a larger set than $\alpha_\cD(\varphi)$. Hence, our definition is weaker than the one provided in~\cite{bauschke2017descent}. In our results, this coefficient will always appear with $\cD=\mathcal{S}$; i.e., the solution set of Problem~\cref{eq:comp_obj}. In particular, we will require that $\alpha_\mathcal{S}(\varphi)>0$, similarly to the condition $\alpha(\varphi)>0$ in previous works~\cite{Bauschke2019,bauschke2017descent,Bolte2018}. 

We note that the proposed restriction to $\cS$ in the above two definitions will turn to be key for proving linear convergence of~\cref{eq:BPGM} within the Kullback--Leibler regression setting we study in~\Cref{sec:KL}.
We are now ready to present the main assumptions of our work.
\begin{assumption}\label{theor assumptions} We assume that $\varphi\in\Gamma_0(\R^N)$ is a Legendre function for which the following assumptions are satisfied.
\begin{enumerate}

\item\label{ass1} $\cC \subseteq \overline{\dom \varphi}$.
\item\label{ass11} $\cS \cap \mathrm{int}(\dom(\varphi)) \neq \emptyset$.
\item\label{ass12} Given $\bx^k \in \cC \cap \mathrm{int}(\dom \varphi)$, $\bx^{k+1}$ in~\eqref{eq:BPGM} is well defined (i.e., it exists, is unique, and belongs to $ \cC \cap \mathrm{int}(\dom \varphi)$).
\item\label{ass2} $F$ is $L$-Lipschitz relative to $\varphi$ for some $L>0$.
\item\label{ass3} $F$ is $\mu$-restricted strongly convex relative to $\varphi$ for some $\mu>0$.
\item\label{ass4} We have $\alpha_\cS(\varphi)>0$.
\end{enumerate}
\end{assumption}
Let us also recall that we already assumed $\cC \subseteq \mathrm{int}(\dom J) = \mathrm{int}(\dom F) \cap \mathrm{int}(\dom R)$. While the first three assumptions above ensure the well-posedness of both \cref{eq:BPGM} and the different tools considered in this work (cf. the Bregman projection and~\Cref{def:alpha(h)}), the remaining three will serve as the key conditions for ensuring a linear convergence rate.

In order to prove our main result, the following lemma is key, and extends~\cite[Lemma 1]{zhang2015} to our Bregman setting.
\begin{lemma}\label{lem:nonsmooth_rrscconseqs} Under~\Cref{theor assumptions}, we have  for all $\bx \in \cC\cap \mathrm{int}(\dom \varphi)$ that 
\begin{equation}\label{eq:nonsmooth_bregman_gradient_domination}
    J(\bx) \geq \bar J+ \mu \alpha_\mathcal{S}(\varphi)D_{\varphi}(\bar\bx_{\varphi},\bx).
\end{equation}
\end{lemma}

\begin{proof} Let us consider the path $\boldsymbol{\gamma}(t)\colon[0,1]\to\R^N$, $\boldsymbol{\gamma}(t):=\bar\bx_{\ell_2}+t(\bx-\bar\bx_{\ell_2})$ and observe that, by~\cite[Section D, Theorem 2.3.4]{HUL2001},
$$
J(\bx)-J(\bar\bx_{\ell_2})= \int_0^1 \langle  s_J(\boldsymbol{\gamma}(t)),\bx-\bar\bx_{\ell_2}\rangle \ \mathrm{d}t,
$$
for all $s_J(\boldsymbol{\gamma}(t))\in\partial J(\boldsymbol{\gamma}(t))$, and where we recall that $\partial J(\bx)$ denotes the subdifferential of $J$ at $\bx$. Now, by~\cite[Theorem 27.4]{Rock1970} there exists $s_J(\bar\bx_{\ell_2})\in\partial J(\bar\bx_{\ell_2})$ such that $\langle s_J(\bar\bx_{\ell_2}), \bx-\bar\bx_{\ell_2}\rangle\geq 0$ (notice that, as $\cC\subseteq\mathrm{int}(\dom J)$, we directly get that $\mathrm{ri}(\dom J)\cap \mathrm{ri}(\cC)\neq\emptyset$, where $\mathrm{ri}(\cD)$ stands for the relative interior of $\cD$). Then, for all $s_J(\boldsymbol{\gamma}(t))\in\partial J(\boldsymbol{\gamma}(t))$ and all $\bx\in\cC \cap \mathrm{int}(\dom \varphi)$ we have that
$$
\begin{aligned}
J(\bx)&= J(\bar\bx_{\ell_2})+\int_0^1 \langle  s_J(\boldsymbol{\gamma}(t)),\bx-\bar\bx_{\ell_2}\rangle \ \mathrm{d}t\\
&= J(\bar\bx_{\ell_2})+\langle s_J(\bar\bx_{\ell_2}), \bx-\bar\bx_{\ell_2}\rangle+\int_0^1 \frac{1}{t}\langle s_J(\boldsymbol{\gamma}(t))-s_J(\bar\bx_{\ell_2}),t(\bx-\bar\bx_{\ell_2})\rangle \ \mathrm{d}t\\
&= J(\bar\bx_{\ell_2})+\langle s_J(\bar\bx_{\ell_2}), \bx-\bar\bx_{\ell_2}\rangle+\int_0^1 \frac{1}{t}\langle s_J(\boldsymbol{\gamma}(t))-s_J(\bar\bx_{\ell_2}),\boldsymbol{\gamma}(t)-\bar\bx_{\ell_2}\rangle \ \mathrm{d}t.
\end{aligned}
$$
Recalling that $J=F+R$, we get that $s_J(\bar\bx_{\ell_2})=\nabla F(\bar\bx_{\ell_2})+s_R(\bar\bx_{\ell_2})$ for some $s_R(\bar\bx_{\ell_2})\in\partial R(\bar\bx_{\ell_2})$ and $s_J(\boldsymbol{\gamma}(t))=\nabla F(\boldsymbol{\gamma}(t))+s_R(\boldsymbol{\gamma}(t))$ for some $s_R(\boldsymbol{\gamma}(t))\in\partial R(\boldsymbol{\gamma}(t))$. Next, since $\partial R$ is a monotone operator (by convexity of $R$), we have that $\langle s_R(\boldsymbol{\gamma}(t)) -  s_R(\bar\bx_{\ell_2}), \boldsymbol{\gamma}(t) - \bar\bx_{\ell_2}\rangle\geq 0$. Hence, the inner product inside the integral above satisfies
$$
\langle s_J(\boldsymbol{\gamma}(t))-s_J(\bar\bx_{\ell_2}),\boldsymbol{\gamma}(t)-\bar\bx_{\ell_2}\rangle\geq \langle \nabla F(\boldsymbol{\gamma}(t))-\nabla F(\bar\bx_{\ell_2}),\boldsymbol{\gamma}(t)-\bar\bx_{\ell_2}\rangle.
$$
Now, as $F$ is restricted strongly convex relative to $\varphi$ with constant $\mu>0$, and using that $\langle s_J(\bar\bx_{\ell_2}), \bx-\bar\bx_{\ell_2}\rangle\geq 0$, we get that
$$
\begin{aligned}
J(\bx)
&\geq J(\bar\bx_{\ell_2})+\int_0^1 \frac{1}{t} \langle \nabla F(\boldsymbol{\gamma}(t))-\nabla F(\bar\bx_{\ell_2}),\boldsymbol{\gamma}(t)-\bar\bx_{\ell_2}\rangle \ \mathrm{d}t\\
&\geq J(\bar\bx_{\ell_2})+\int_0^1 \frac{1}{t}\mu \langle \nabla \varphi(\boldsymbol{\gamma}(t))-\nabla\varphi(\bar\bx_{\ell_2}),\boldsymbol{\gamma}(t)-\bar\bx_{\ell_2}\rangle \ \mathrm{d}t\\
&=J(\bar\bx_{\ell_2})+\mu\left(\int_0^1 \langle \nabla \varphi(\boldsymbol{\gamma}(t)), \bx-\bar\bx_{\ell_2} \rangle \ \mathrm{d}t - \langle \nabla \varphi(\bar\bx_{\ell_2}), \bx-\bar\bx_{\ell_2}\rangle\right).
\end{aligned}
$$
Note that, in order to obtain the second inequality, we used the fact that for all $t \in [0,1]$ we have 
$\overline{\boldsymbol{\gamma}}_{\ell_2}(t) = \bar{\bx}_{\ell_2}$ 
by the definition of both $\boldsymbol{\gamma}(t)$ and the Euclidean projection. Applying the fundamental theorem of calculus to $\varphi$, together with $J(\bar\bx_{\ell_2})=\bar J$, leads to
\begin{equation}\label{eq:breg_fuc_growth}
J(\bx)\geq \bar J+\mu D_{\varphi}(\bx,\bar\bx_{\ell_2}),
\end{equation}
for all $\bx\in\cC\cap \mathrm{int}(\dom \varphi)$. Next, since $\alpha_\mathcal{S}(\varphi)>0$ by assumption, we have that $\alpha_\mathcal{S}(\varphi)D_{\varphi}(\bar\bx_{\ell_2},\bx)\leq D_{\varphi}(\bx,\bar\bx_{\ell_2})$ and, by definition of $\bar\bx_{\varphi}$, we obtain that $D_{\varphi}(\bar\bx_{\ell_2},\bx)\geq D_{\varphi}(\bar\bx_{\varphi},\bx)$. Plugging both of these facts into~\cref{eq:breg_fuc_growth} finally leads to
$$
J(\bx)\geq \bar J+\mu\alpha_\mathcal{S}(\varphi) D_{\varphi}(\bar\bx_{\ell_2},\bx)\geq \bar J+\mu\alpha_\mathcal{S}(\varphi) D_{\varphi}(\bar\bx_{\varphi},\bx),
$$
for all $\bx\in\cC\cap \mathrm{int}(\dom \varphi)$, as desired.
\end{proof}

The following theorem constitutes the main result of the section, and shows linear convergence of~\cref{eq:BPGM}.
\begin{theorem}[Linear convergence of BPGM]\label{thm:lin_conv_rate}
Let~\Cref{theor assumptions} be satisfied and consider the sequence $(\bx^k )_{k\in\mathbb{N}}$ in $\cC\cap \mathrm{int}(\dom \varphi)$ generated by~\cref{eq:BPGM} with step-size $\tau\in(0,1/L]$. Then, 
\begin{enumerate}[label=(\roman*),leftmargin=*]
\item\label{item1mainthm} The Bregman divergence between $(\bx^k)_{k\in\mathbb{N}}$ and  $\cS$ decreases linearly; i.e.,
\begin{equation}\label{eq:linear_conv_full_rank}
D_{\varphi}(\bar\bx^{k+1}_{\varphi},\bx^{k+1})\leq \left(\frac{1}{1+\alpha_\mathcal{S}(\varphi)\tau\mu}\right)^k D_{\varphi}(\bar\bx^{0}_{\varphi},\bx^{0}),
\end{equation}
where, according to~\cref{eq:bregproj}, $\bar \bx^{k+1}_{\varphi}$ (resp. $\bar \bx_{\varphi}^0$) denotes the Bregman projection of $\bx^{k+1}$ (resp. $\bx^0$) onto $\mathcal{S}$.
\item\label{item2mainthm} The sequence $(J(\bx^k))_{k\in\mathbb{N}}$ converges linearly to $\bar J$; that is,
\begin{equation}\label{eq:linconv func values}
J(\bx^{k+1})-\bar J\leq \frac{1}{\tau}\left(\frac{1}{1+\alpha_\mathcal{S}(\varphi)\tau\mu}\right)^kD_{\varphi}(\bar\bx^0_{\varphi},\bx^0).
\end{equation}
\end{enumerate}
\end{theorem}

\begin{proof}  We start by proving~\cref{item1mainthm}. 
From the optimality condition of the~\eqref{eq:BPGM} iteration we get 
\begin{equation}
    \nabla \varphi(\bx^k) - \nabla \varphi (\bx^{k+1}) - \tau \nabla F(\bx^k) \in \tau \partial R(\bx^{k+1}),
\end{equation}
which, combined with the subgradient inequality for the convex function $R$ leads to, for all $\bx \in \cC\cap \mathrm{int}(\dom \varphi) $,
$$\tau R(\bx) \geq \tau R(\bx^{k+1}) + \langle\nabla \varphi(\bx^k) - \nabla \varphi (\bx^{k+1}) -  \tau \nabla F(\bx^k), \bx - \bx^{k+1} \rangle \notag. $$
Rearranging the terms we get that
\begin{equation}\label{eq:subgrad_ineq}
         \tau( R(\bx^{k+1}) - R(\bx)) \leq \tau \langle \nabla F(\bx^k), \bx - \bx^{k+1} \rangle  + \langle\nabla \varphi(\bx^{k+1}) - \nabla \varphi (\bx^{k}), \bx - \bx^{k+1} \rangle .
\end{equation}
We now introduce the following three (in)equalities:
\begin{itemize}
    \item  \textit{Convexity of $F$}:
    \begin{align*}
        \langle \nabla F(\bx^k), \bx - \bx^{k+1} \rangle &= \langle \nabla F(\bx^k), \bx - \bx^{k} \rangle + \langle \nabla F(\bx^k), \bx^k - \bx^{k+1} \rangle \\
        & \leq F(\bx) - F(\bx^{k})  + \langle \nabla F(\bx^k), \bx^k - \bx^{k+1} \rangle .
    \end{align*}
    \item \textit{Three points identity~\cite{ChenTeboulle1993}}:
    \begin{equation*}
        \langle \nabla \varphi(\bx^{k+1}) - \nabla \varphi(\bx^k), \bx - \bx^{k+1} \rangle = D_\varphi(\bx,\bx^k) - D_\varphi(\bx,\bx^{k+1}) - D_\varphi(\bx^{k+1},\bx^k) .
    \end{equation*}
    \item \textit{Relative $L$-smoothness of $F$~\eqref{eq:rel_Lip}}: for all $\tau \leq 1/L$
    \begin{align}
        & F(\bx^{k+1}) \leq F(\bx^k) + \langle \nabla F(\bx^k), \bx^{k+1} - \bx^k\rangle + \frac{1}{\tau}D_\varphi (\bx^{k+1},\bx^k) \label{eq:rel_smooth_proof} \\
        \Longleftrightarrow \; & \langle \nabla F(\bx^k), \bx^{k} - \bx^{k+1}\rangle \leq F(\bx^k) - F(\bx^{k+1})  + \frac{1}{\tau}D_\varphi (\bx^{k+1},\bx^k). \notag  
    \end{align}    
\end{itemize}
Injecting these equations in~\eqref{eq:subgrad_ineq} we obtain
$$
\begin{aligned}
    \tau( R(\bx^{k+1}) - R(\bx)) &\leq \tau (F(\bx) - F(\bx^{k})) + \tau (F(\bx^k) - F(\bx^{k+1})) + D_\varphi (\bx^{k+1},\bx^k) \\
    &\quad + D_\varphi(\bx,\bx^k) - D_\varphi(\bx,\bx^{k+1}) - D_\varphi(\bx^{k+1},\bx^k),
\end{aligned}
$$
which simplifies as
$$
D_\varphi(\bx,\bx^{k+1}) \leq D_\varphi(\bx,\bx^k) + \tau (J(\bx) - J(\bx^{k+1})). 
$$
Note that we could have reached this point more directly using~\cite[Lemma 5]{bauschke2017descent}; however, the above details will serve to easily prove the convergence result of BPGM with backtracking stated in Corollary~\ref{coro:lin_cv_backtrack}. 

Next, taking $\bx=\bar\bx^k_{\varphi}$ and using that $D_{\varphi}(\bar \bx^{k+1}_{\varphi}, \bx^{k+1})\leq D_{\varphi}(\bar \bx^{k}_{\varphi}, \bx^{k+1})$ by  definition of $\bar\bx^{k+1}_{\varphi}$, we have that (notice that $J(\bar\bx^k_{\varphi})=J(\bar\bx^{k+1}_{\varphi})= \bar J)$
\begin{equation}\label{eq:nonsmooth_jerome_lemma}
D_{\varphi}(\bar \bx^{k+1}_{\varphi}, \bx^{k+1})\leq D_{\varphi}(\bar \bx^{k}_{\varphi}, \bx^{k+1})\leq D_{\varphi}(\bar\bx^k_{\varphi},\bx^k)+\tau (\bar J - J(\bx^{k+1})).
\end{equation}
Now, by~\Cref{lem:nonsmooth_rrscconseqs} we get that
\begin{equation}\label{eq:alphaimplication_composite}
J(\bx^{k+1})\geq \bar J + \mu \alpha_\mathcal{S}(\varphi)D_{h}(\bar\bx^{k+1}_{\varphi},\bx^{k+1}).
\end{equation}
Finally, combining~\cref{eq:alphaimplication_composite} with~\cref{eq:nonsmooth_jerome_lemma}, we derive that
\begin{equation}\label{eq:non_sym_lin_cv_composite}
\tau\mu \alpha_\mathcal{S}(\varphi) D_{\varphi}(\bar\bx^{k+1}_{\varphi}, \bx^{k+1})+D_{\varphi}(\bar \bx^{k+1}_{\varphi}, \bx^{k+1})\leq  D_{\varphi}(\bar\bx^k_{\varphi},\bx^k)    
\end{equation}
which in turn implies that
$$
D_{\varphi}(\bar \bx^{k+1}_{\varphi}, \bx^{k+1})\leq \frac{1}{1+\alpha_\mathcal{S}(\varphi)\tau\mu}D_{\varphi}(\bar \bx^k_{\varphi}, \bx^k),
$$
and applying the above inequality recursively finally leads to
\begin{equation}\label{eq: lin conv D varphi}
    D_{\varphi}(\bar\bx^{k+1}_{\varphi},\bx^{k+1})\leq \left(\frac{1}{1+\alpha_\mathcal{S}(\varphi)\tau\mu}\right)^k D_{\varphi}(\bar\bx^{0}_{\varphi},\bx^{0}),
\end{equation}
as desired. 

We will now show~\cref{item2mainthm}. Adding $R(\bx^{k+1})$ to both sides of the relative smoothness inequality \eqref{eq:rel_smooth_proof} we get that
$$
\begin{aligned}
J(\bx^{k+1})&\leq F(\bx^k)+\langle \nabla F(\bx^k),\bx^{k+1}-\bx^k\rangle+\frac{1}{\tau}D_{\varphi}(\bx^{k+1},\bx^k)+R(\bx^{k+1})\\
&=\min_{\bx\in\cC} F(\bx^k)+\langle \nabla F(\bx^k),\bx-\bx^k\rangle+\frac{1}{\tau}D_{\varphi}(\bx,\bx^k)+R(\bx)\\
&\leq \min_{\bx\in\cC} F(\bx)+\frac{1}{\tau}D_{\varphi}(\bx,\bx^k)+R(\bx)= \min_{\bx\in\cC} J(\bx)+\frac{1}{\tau}D_{\varphi}(\bx,\bx^k),
\end{aligned}
$$
where the first equality is obtained by the definition of~\eqref{eq:BPGM} and  the last inequality comes from the fact that $F$ is convex. Finally, we have that
$$
J(\bx^{k+1})\leq \min_{\bx\in\cC} J(\bx)+\frac{1}{\tau}D_{\varphi}(\bx,\bx^k)\leq \bar J+\frac{1}{\tau}D_{\varphi}(\bar\bx^k_{\varphi},\bx^k),
$$
which, by~\cref{eq: lin conv D varphi}, leads to
\begin{equation}\label{eq:rate_func_val}
    J(\bx^{k+1})-\bar J\leq \frac{1}{\tau}D_{\varphi}(\bar\bx^k_{\varphi},\bx^k)\leq \frac{1}{\tau}\left(\frac{1}{1+\alpha_\mathcal{S}(\varphi)\tau\mu}\right)^kD_{\varphi}(\bar\bx^0_{\varphi},\bx^0),
\end{equation}
completing the proof of the result.
\end{proof}
Note that, if $\mathcal{S}=\{\bar\bx\}$; i.e., if the solution to~\cref{eq:comp_obj} is unique, then~\cref{eq:linconv func values} simply writes
$$
J(\bx^{k+1})-\bar J\leq \frac{1}{\tau}\left(\frac{1}{1+\alpha_\mathcal{S}(\varphi)\tau\mu}\right)^kD_{\varphi}(\bar\bx,\bx^0).
$$

\begin{remark} We make three remarks regarding~\Cref{thm:lin_conv_rate}:
\begin{enumerate}
\item[(a)] While in \Cref{theor assumptions} we require $F$ to satisfy~\cref{eq:rrsc}, it is worth noting that, when $R$ is differentiable, this condition can instead be imposed on the composite functional $J$. In this way,~\cref{eq:rrsc} may be inherited either by $J$ itself or solely by one of the two terms $F$ or $R$. A generalization of~\cref{eq:rrsc} to non-differentiable functions would also open the possibility to leverage the geometrical properties of non-differentiable $R$ in order to satisfy the required assumptions. 
\item[(b)] 
Note that Theorem~\ref{thm:lin_conv_rate} does not guarantee global pointwise convergence of the iterates in general (i.e., that the iterates $(\bx^k)_{k\in\mathbb{N}}$ generated by~\cref{eq:BPGM} converge to a solution of~\cref{eq:comp_obj}). This may be achieved by requiring additional assumptions on the Legendre function $\varphi$ such as for instance Assumption $\mathbf{H}$ in~\cite{bauschke2017descent, Teboulle2018}.

\item[(c)] Notice that, although the result of~\Cref{thm:lin_conv_rate} is stated in terms of the Bregman projection associated to $\varphi$ (i.e., $\bar \bx_{\varphi}$) our~\cref{eq:rrsc} depends instead on the Euclidean projection (i.e., $\bar\bx_{\ell_2}$). This difference of geometry is primarily technical for two reasons. First, while the Bregman projection arises naturally in the proof of~\Cref{thm:lin_conv_rate}, proving the growth condition of~\Cref{lem:nonsmooth_rrscconseqs} requires that all points along the line path $\boldsymbol{\gamma}(t)$ project onto the same point, namely an Euclidean projection. The key inequality $D_{\varphi}(\bar\bx_{\ell_2},\bx)\geq D_{\varphi}(\bar\bx_{\varphi},\bx)$ allows us to complete the proof. Second, using an Euclidean projection in~\cref{eq:rrsc} is key in the proof of existence of a $\mu>0$ for functions of the form $F=G(\bA\cdot)$, as it is the natural projection to consider when exploiting the decomposition $\R^N = \mathrm{ker}(\bA) \oplus \mathrm{ker}(\bA)^\perp$ (see~\Cref{prop:rrsc for KL} for an example).
\end{enumerate}
\end{remark}

In practice, deriving an accurate estimate of the relative smoothness constant $L$ can be challenging, often resulting in the selection of step-sizes that are smaller than the optimal valid value. One way to circumvent this issue,  and significantly ease the practical deployment of~\cref{eq:BPGM}, is to use a backtracking strategy. Such strategies not only automatically adjust the step-size but also allow for larger step-sizes in the early iterations, thereby further accelerating convergence. In fact, the linear convergence result of~\Cref{thm:lin_conv_rate} can be extended to~\eqref{eq:BPGM} using such a step-size selection strategy, as stated by the following corollary. 

\begin{corollary}[Linear convergence of BPGM with backtracking]\label{coro:lin_cv_backtrack}
    Let Assumption~\ref{theor assumptions} be satisfied, and consider the sequence $(\bx^k)_{k \in \mathbb{N}}$ generated by~\cref{eq:BPGM} with step-sizes $\tau^k$ obtained via the following backtracking strategy: for $\beta \in (0,1)$, and $\tau^0 > 0$, we set, for all $k\in\mathbb{N}$, $\tau^k := \tau^{k-1} \cdot \beta^j$ for the smallest $j=0,1,\ldots$ such that $\bx^{k+1} \in \cC\cap \mathrm{int}(\dom \varphi)$ and 
    \begin{equation}\label{eq:backtrack_cond}
        F(\bx^{k+1}) \leq F(\bx^k) + \langle \nabla F(\bx^k), \bx^{k+1} - \bx^k\rangle + \frac{1}{\tau^k}D_\varphi (\bx^{k+1},\bx^k) 
    \end{equation}
    holds true.
    
    Then, the sequence $(\bx^k)_{k \in \mathbb{N}}$ obeys the two statements of Theorem~\ref{thm:lin_conv_rate} by replacing $\tau$ with $\bar{\tau} = \beta / L$. 
\end{corollary}
\begin{proof}
First of all, one can reproduce the proof of Theorem~\ref{thm:lin_conv_rate} replacing $\tau$ by $\tau^k$ and the relative $L$-smoothness inequality in~\eqref{eq:rel_smooth_proof} by the backtracking condition~\eqref{eq:backtrack_cond}. We then get, in place of~\eqref{eq: lin conv D varphi}, the following inequality
\begin{equation}
    D_{\varphi}(\bar\bx^{k+1}_{\varphi},\bx^{k+1})\leq \prod_{l=1}^k \left(\frac{1}{1+\alpha_\mathcal{S}(\varphi)\tau^l\mu}\right) D_{\varphi}(\bar\bx^{0}_{\varphi},\bx^{0}),
\end{equation}
and similarly for~\eqref{eq:rate_func_val}. It remains to show that the sequence of step-sizes $(\tau^k)_{k \in \mathbb{N}}$ does not vanish. To do so, we will show that there  exists $\bar \tau >0$ such that  $ \tau^k \geq \bar \tau$ for all $k$, as in that case
$$
\prod_{l=1}^k \left(\frac{1}{1+\alpha_\mathcal{S}(\varphi)\tau^l\mu}\right) \leq \left(\frac{1}{1+\alpha_\mathcal{S}(\varphi)\bar \tau\mu}\right)^k,
$$
thus recovering the linear convergence rate. From the relative $L$-smoothness inequality~\eqref{eq:rel_smooth_proof} we get that the backtracking condition~\eqref{eq:backtrack_cond}
is always guaranteed whenever $\tau^k \leq 1/L$. Next, we claim that $\tau^k>\bar\tau:=\beta/L$. We will proceed by contradiction assuming that $\tau^k \leq  {\beta}/{L}$. By recalling the step-size selection rule $\tau^k = \tau^{k-1} \cdot \beta^j$, we get that $\tau^k / \beta = \tau^{k-1} \cdot \beta^{j-1} \leq {1}/{L}$, and~\eqref{eq:backtrack_cond} holds for $\tau^{k-1} \cdot \beta^{j-1}$. This contradicts the minimality of $j$, and completes the proof of the result.
\end{proof}

Deploying this backtracking strategy only requires the selection of an arbitrary initial value $\tau^0$ and $\beta \in(0,1)$. Moreover, the computational overhead it introduces is negligible. Indeed, the gradient $\nabla F(\bx^k)$ still needs to be computed only once per iteration, just as in the constant step-size variant, and not once per subiteration of the backtracking search. Hence, the additional computational cost is limited to only a few evaluations of the term $D_\varphi(\bx^{k+1}, \bx^k)$.\\






\paragraph{\textbf{Relation to existing results.}}
As mentioned in the introduction, linear convergence of~\cref{eq:BPGM} has been previously analyzed in~\cite{Bauschke2019, lu2018relatively, Teboulle2018}, under different assumptions. In this paragraph, we discuss how our results are related with and complement these  works.

First, we extend~\cite{lu2018relatively,Teboulle2018} to functions that are not \emph{globally}  strongly convex relative to $\varphi$, thanks to the proposed~\cref{eq:rrsc}, which only requires relative strong convexity between a point and its projection onto the solution set. However, this comes at the price of introducing the symmetry coefficient $\alpha_\mathcal{S}(\varphi)$ and a worse convergence rate than the one derived in~\cite{lu2018relatively,Teboulle2018} when global relative strong convexity holds. 

Second, while both our~\cref{eq:rrsc} condition and the \emph{gradient dominated conditions} of~\cite[Definition 3.2]{Bauschke2019} are weaker than relative strong convexity, they lead to rather different convergence analyses, and it is difficult to compare the resulting rates. In particular, the derived linear convergence rate in~\cite{Bauschke2019}  relies on the existence of a parameter $\theta(\tau)$, which, as the authors acknowledge, may be difficult to check in practice. \\

We are now ready to apply our theoretical framework to (regularized) Kullback--Leibler regression problems.

\section{Applications: Kullback--Leibler regression}\label{sec:KL}
We consider (regularized) Kullback--Leibler (KL) regression problems, that is Problem~\cref{eq:comp_obj} with $\cC=\R^N_{\geq 0}$, $R\in\Gamma_0(\cC)$ a bounded from below regularizer, and $F$ the generalized, smoothed Kullback-Leibler divergence, given by
\begin{equation}\label{eq:fid_KL}
F(\bx)= \mathrm{KL}(\by, \bA\bx+\mathbf{b}) := \sum_{m=1}^M [\bA\bx]_m + b_m + y_m \log \left(\frac{y_m}{[\bA\bx]_m + b_m} \right) - y_m,
\end{equation}
where $\bA\in \R_{\geq 0}^{M \times N}$ is a linear operator modeling some measurement acquisition process, $\mathbf{b}\in \R^M_{> 0}$ a background signal (a commonly-used additive term in practice, see e.g.~\cite{harmany2011spiral}), and $\by\in \mathbb{R}^M_{> 0}$ denotes the vector of observations. We additionally require that $\bA$ does not have any column entirely zero. The above optimization problem encompasses several applications such as Poisson inverse problems, particularly in astronomical imaging~\cite{Bertero2009,Bertero2018} or microscopy~\cite{dey20043d, Zunino_2023}. It is also at the heart of Poisson nonnegative matrix factorization~\cite{LS2000}, with applications in latent semantic analysis \cite{can04}, recommender systems~\cite{gopalan2015scalable} or spectral unmixing~\cite{Fevotte2009}. Finally, notice that we call~\cref{eq:fid_KL} generalized, smoothed KL divergence due to the fact that $\bA\bx$, $\by$ may not necessarily be discrete probability distributions, and because we account for the presence of the background term $\mathbf{b}>0$. However, for conciseness, we will call~\eqref{eq:fid_KL} the KL divergence.


\begin{remark}\label{rem:dark_regime}
While we assume $\by>0$ for simplicity, it is worth mentioning that the upcoming analysis also holds for $\by \geq 0$. Indeed, when $y_m=0$, the terms in~\cref{eq:fid_KL} with $y_m=0$ (using the common convention $0 \log 0 = 0$) are linear and can be seen as an analysis-like $\ell_1$-norm regularizer, while the remaining terms are rewritten as a new KL term with $\tilde{\by} > 0$. Specifically, let us denote by
 $\sigma(\by)\subset\{1,\ldots, M\}$ the support of $\by$ and  $\bA_{\sigma(\by)}$ the restriction of $\bA$ to the rows indexed by $\sigma(\by)$. By recalling that both $\bx$ and $\bA$ have nonnegative entries, we have that
$$
\begin{aligned}
\mathrm{KL}(\by,\bA\bx + \mathbf{b})&= \mathrm{KL}(\by_{\sigma(\by)}, \bA_{\sigma(\by)}\bx + \mathbf{b}_{\sigma(\by)})+\sum_{m\in\sigma(\by)^c}\left([\bA\bx]_m + b_m\right)\\
&=\mathrm{KL}(\by_{\sigma(\by)},\bA_{\sigma(\by)}\bx + \mathbf{b}_{\sigma(\by)})+\|\bA_{\sigma(\by)^c}\bx\|_1+\|\mathbf{b}_{\sigma(\by)^c}\|_1
\end{aligned}
$$
and~\cref{eq:comp_obj} simply becomes
$$
\argmin_{\bx\in\cC} \mathrm{KL}(\by_{\sigma(\by)},\bA_{\sigma(\by)}\bx + \mathbf{b}_{\sigma(\by)})+\|\bA_{\sigma(\by)^c}\bx\|_1+R(\bx),
$$
as desired.
\end{remark}


In the following, we analyze the convergence of~\eqref{eq:BPGM} for such KL problems from both theoretical (\Cref{sec:KL_theory}) and numerical (\Cref{sec:experiments}) perspectives. This is done for the three following Bregman functions $\varphi$:
\begin{enumerate}
\item\label{varphi is Burg} \emph{Burg's entropy}: $\varphi_\mathrm{B}(\bx)=\sum_{n=1}^N-\log(x_n)$, with $\dom\varphi_\mathrm{B}= \R_{>0}^N$,
\item\label{varphi is sBurg} \emph{Smoothed Burg's entropy}: $\varphi_{\mathrm{sB}}(\bx)=\sum_{n=1}^N -\log(x_n+\xi)$ for some smoothing parameter $\xi>0$, and where $\dom\varphi_{\mathrm{sB}}=\R_{> -\xi}^N$, 
\item\label{varphi is l2} \emph{Squared $\ell_2$-norm}: $\varphi_{\ell_2}(\bx)=(1/2)\|\bx\|_2^2$, with $\dom\varphi_{\ell_2}=\R^N$.
\end{enumerate}
While the Burg's entropy is commonly used for KL regression, since the KL divergence is relatively smooth with respect to $\varphi_\mathrm{B}$, see~\cite{bauschke2017descent} and~\Cref{prop:rel_smooth_KL} (with a better constant $L$ than for $\varphi_{\ell_2}$), the smoothed Burg's entropy has, to our knowledge, never been considered in this context. Yet, as we will show, it corresponds to the suitable geometry required to achieve linear convergence. Although not previously used in our present context, such a function has already been considered as a Bregman function for constructing exact continuous relaxations of the $\ell_0$ pseudo-norm~\cite{EssafriKL,Essafri2024}, for demonstrating the existence of an RRSC-like constant in Poisson regression problems~\cite{chirinos2025}, for designing an Expectation-Majorization algorithm for Penalized Maximum-Likelihood image reconstruction problems~\cite{FessHe1995}, and, very recently, as a tangent majorant for Bregman Majorization-Minimization algorithms~\cite{adly2026}, where the authors refer to such $\varphi$ as a ``log-shift'' function.

\subsection{Theoretical analysis}\label{sec:KL_theory}
To start, we provide some relevant properties about the solution set $\cS$ of~\cref{eq:comp_obj} in the present context of KL regression.

\begin{lemma}\label{lem:sol set KL} The solution set $\cS$ of~\cref{eq:comp_obj}, with $F$ given by~\cref{eq:fid_KL} and any $R\in\Gamma_0(\R_{\geq 0}^N)$ bounded from below, is nonempty, closed, and bounded (hence compact).
\end{lemma}

\begin{proof} To prove the result we only need to show that the composite function $J$ is coercive on $\R_{\geq 0}^N$. Indeed, given that $J$ is lower semi-continuous over the closed convex set $\R_{\geq 0}^N$, we get  by \cite[Proposition 11.15]{BCombettes} that coercivity ensures that $\mathcal{S}$ is nonempty and compact. To that end, we will show that, if $\|\bx\|_1\to+\infty$, then $J(\bx)\to+\infty$. Recall that, by assumption, the matrix $\bA\in\R^{M\times N}_{\geq 0}$, with entries $a_{m,n}$, $m=1,\ldots,M$, $n=1,\ldots, N$, has no column which is entirely zero. This means that, for every $n\in\{1,\ldots, N\}$, there exists $m\in\{1,\ldots, M\}$ such that $a_{m,n}>0$. Then,  defining the constant $c := \min_n \|\ba_n\|_1$, we have that 
$$
\|\bA\bx\|_1=\sum_{m=1}^M \left|\sum_{n=1}^N a_{m,n}x_n\right| = \sum_{n=1}^N x_n \sum_{m=1}^M a_{m,n} \geq c\|\bx\|_1,
$$
for all $\bx\in\cC$. Consequently, if $\|\bx\|_1\to+\infty$, then indeed $\|\bA\bx\|_1\to+\infty$. Now, since $\mathrm{KL}(\by,\cdot+\mathbf{b})$ is coercive, we derive that
$$
\lim_{\|\bx\|_1\to+\infty} \mathrm{KL}(\by, \bA\bx+\mathbf{b})=\lim_{\bz\to+\infty}\mathrm{KL}(\by,\bz+\mathbf{b})=+\infty,
$$
showing that $\mathrm{KL}(\by, \bA\cdot+\mathbf{b})$ is coercive. Finally, combining this with the fact that $R$ is bounded from below proves that $J$ is coercive on $\R_{\geq 0}^N$ and completes the proof. 
\end{proof}

We now aim at verifying whether~\Cref{theor assumptions} holds for each of the three functions $\varphi$ we consider. First, one can see from their definitions that we  have $\cC\subseteq\overline{\dom\varphi}$, that is \Cref{theor assumptions}-\ref{ass1}. Second, \Cref{theor assumptions}-\cref{ass11,ass12}, are trivially satisfied for $\varphi_{\mathrm{sB}}$ and $\varphi_{\ell_2}$ as $\cS \subset \cC$ by construction and $\cC\subset \mathrm{int}(\dom \varphi_{\mathrm{sB}}) \subset \mathrm{int}(\dom \varphi_{\ell_2})$. Regarding $\varphi_\mathrm{B}$, \Cref{theor assumptions}-\ref{ass12} holds as $D_{\varphi_\mathrm{B}}(\cdot, \bx^k)$ in~\eqref{eq:BPGM} acts as a log-barrier when approaching $\partial \cC$, and \Cref{theor assumptions}-\ref{ass11} restricts our analysis for $\varphi_\mathrm{B}$ to situations in which $\cS$ contains at least one  positive (i.e., without any zero component) solution. To analyze the three final requirements in \Cref{theor assumptions}, which are both the most involved and the most important, we introduce the following set: given $\theta\geq0$, define
\begin{equation}\label{eq:Skerperp}
    \cS_{\ker}^{\theta}
    :=  \left\{
        \bx \in \cC \setminus \cS
        \ \mid \
        \angle(\bx - \bar \bx_{\ell_2},\ker(\bA))\leq\theta
    \right\}. 
\end{equation}
In words, the set $\cS_{\ker}^{\theta}$ consists of all non-solution points whose normal displacement from their Euclidean projection onto $\cS$ forms a sufficiently small angle with $\ker(\bA)$.
The role and importance of this set is revealed by the following lemma.
\begin{lemma}\label{lem:necessary condition}
  Let $F$ be the KL divergence given by~\cref{eq:fid_KL}. Then, regardless of the choice of $\varphi$, we have $\mu = 0$ in~\eqref{eq:rrsc} for all $\bx \in \cS^0_{\ker}$.
\end{lemma}
\begin{proof} To prove the result, we proceed by contradiction; i.e., by assuming that there exists $\mu>0$ such that, for all $\bx\in \cS^0_{\ker}$,~\eqref{eq:rrsc} holds. However, given the form of $F$, the left hand side can be expressed as
$$
\begin{aligned}
\langle \nabla F(\bx)-\nabla F(\bar\bx_{\ell_2}), \bx-\bar\bx_{\ell_2}\rangle 
&= 
\langle \bA^T\nabla \mathrm{KL}(\by,\bA\bx)- \bA^T\nabla \mathrm{KL}(\by,\bA\bar\bx_{\ell_2}),\bx-\bar\bx_{\ell_2} \rangle\\
&=
\langle \nabla \mathrm{KL}(\by,\bA\bx)- \nabla \mathrm{KL}(\by,\bA\bar\bx_{\ell_2}),\bA(\bx-\bar\bx_{\ell_2}) \rangle
\end{aligned}
$$
As $\bx\in \cS^0_{\ker}$, we get that $\bx-\bar\bx_{\ell_2}\in\ker(\bA)$, making the above term equal to zero, and hence, the existence of such $\mu>0$ impossible.
\end{proof}

From Lemma~\ref{lem:necessary condition}, we deduce that for~\eqref{eq:rrsc} to hold over $\cC$ for the KL divergence, a necessary condition is that $\cS_{\ker}^{0} = \emptyset$. This is satisfied, for instance, when $\cS + \ker(\bA) = \cS$. However, $\cS_{\ker}^{0} \neq \emptyset$ in many cases of practical interest such as when $\bA$ is rank-deficient and $R \neq 0$ is such that the problem admits a unique solution. As such, to gain insights into most KL regression problems, we study in the following proposition, for each considered $\varphi$, the subset of $\cC$ over which~\eqref{eq:rrsc} holds.

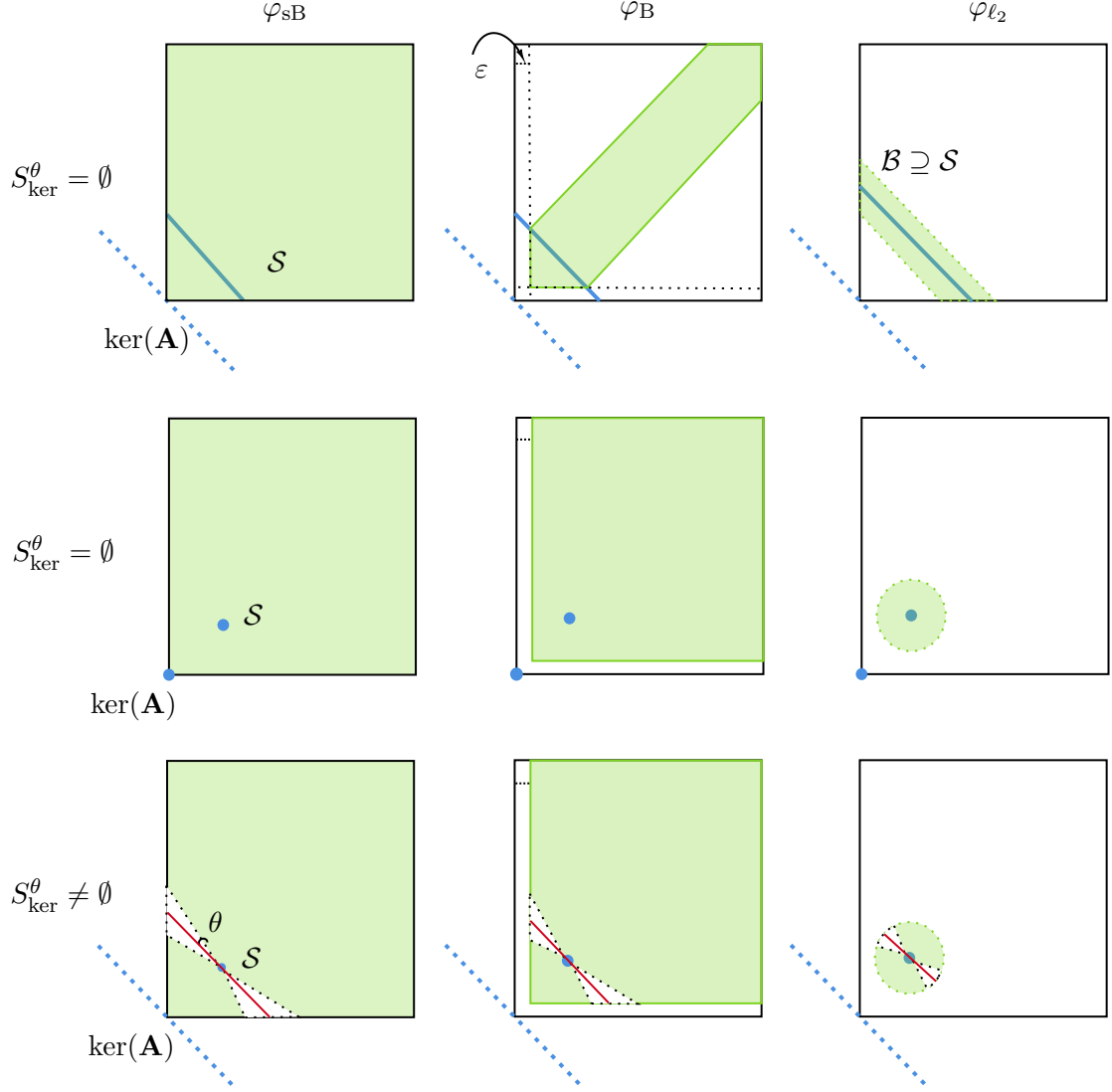
\begin{figure}[htbp]
    \centering

\tikzset{every picture/.style={line width=0.75pt}} 

\begin{tikzpicture}[x=0.65pt,y=0.65pt,yscale=-1,xscale=1]

\draw [color={rgb, 255:red, 74; green, 144; blue, 226 }  ,draw opacity=1 ][fill={rgb, 255:red, 74; green, 144; blue, 226 }  ,fill opacity=1 ][line width=1.5]  [dash pattern={on 1.69pt off 2.76pt}]  (56.81,139.14) -- (137.02,222.38) ;
\draw [color={rgb, 255:red, 74; green, 144; blue, 226 }  ,draw opacity=1 ][line width=1.5]    (95.75,129.38) -- (140.25,179.38) ;
\draw  [fill={rgb, 255:red, 126; green, 211; blue, 33 }  ,fill opacity=0.28 ][line width=0.75]  (95.48,30.38) -- (239.19,30.38) -- (239.19,179.5) -- (95.48,179.5) -- cycle ;
\draw   (298.25,30.39) -- (441.99,30.39) -- (441.99,179.55) -- (298.25,179.55) -- cycle ;
\draw   (499.19,30.39) -- (642.93,30.39) -- (642.93,179.55) -- (499.19,179.55) -- cycle ;
\draw [color={rgb, 255:red, 74; green, 144; blue, 226 }  ,draw opacity=1 ][fill={rgb, 255:red, 74; green, 144; blue, 226 }  ,fill opacity=1 ][line width=1.5]  [dash pattern={on 1.69pt off 2.76pt}]  (258.14,137.93) -- (338.36,221.17) ;
\draw [color={rgb, 255:red, 74; green, 144; blue, 226 }  ,draw opacity=1 ][fill={rgb, 255:red, 74; green, 144; blue, 226 }  ,fill opacity=1 ][line width=1.5]  [dash pattern={on 1.69pt off 2.76pt}]  (459.08,137.93) -- (539.29,221.17) ;
\draw [color={rgb, 255:red, 74; green, 144; blue, 226 }  ,draw opacity=1 ][line width=1.5]    (298.25,128.63) -- (342.51,174.3) -- (347.72,179.68) ;
\draw [color={rgb, 255:red, 74; green, 144; blue, 226 }  ,draw opacity=1 ][line width=1.5]    (499.13,112.97) -- (564.56,179.9) ;
\draw  [color={rgb, 255:red, 126; green, 211; blue, 33 }  ,draw opacity=1 ][fill={rgb, 255:red, 126; green, 211; blue, 33 }  ,fill opacity=0.28 ] (441.99,30.39) -- (441.63,63.16) -- (340.53,171.95) -- (307.45,171.83) -- (307.42,137.47) -- (411.08,30.16) -- cycle ;
\draw  [color={rgb, 255:red, 126; green, 211; blue, 33 }  ,draw opacity=1 ][fill={rgb, 255:red, 126; green, 211; blue, 33 }  ,fill opacity=0.28 ][dash pattern={on 0.84pt off 2.51pt}] (499.33,97.33) -- (578.05,179.53) -- (545.82,179.53) -- (499.3,129.14) -- cycle ;
\draw  [fill={rgb, 255:red, 126; green, 211; blue, 33 }  ,fill opacity=0.27 ] (96.98,248.05) -- (240.69,248.05) -- (240.69,397.17) -- (96.98,397.17) -- cycle ;
\draw   (299.32,247.63) -- (443.06,247.63) -- (443.06,396.79) -- (299.32,396.79) -- cycle ;
\draw   (500.26,247.63) -- (644,247.63) -- (644,396.79) -- (500.26,396.79) -- cycle ;
\draw  [color={rgb, 255:red, 126; green, 211; blue, 33 }  ,draw opacity=1 ][fill={rgb, 255:red, 126; green, 211; blue, 33 }  ,fill opacity=0.29 ] (308.52,247.79) -- (443.15,247.79) -- (443.15,389.08) -- (308.52,389.08) -- cycle ;
\draw  [color={rgb, 255:red, 74; green, 144; blue, 226 }  ,draw opacity=1 ][fill={rgb, 255:red, 74; green, 144; blue, 226 }  ,fill opacity=1 ] (327.39,364.3) .. controls (327.39,365.91) and (328.65,367.22) .. (330.2,367.22) .. controls (331.75,367.22) and (333.01,365.91) .. (333.01,364.3) .. controls (333.01,362.69) and (331.75,361.39) .. (330.2,361.39) .. controls (328.65,361.39) and (327.39,362.69) .. (327.39,364.3) -- cycle ;
\draw  [color={rgb, 255:red, 74; green, 144; blue, 226 }  ,draw opacity=1 ][fill={rgb, 255:red, 74; green, 144; blue, 226 }  ,fill opacity=1 ] (296.18,396.79) .. controls (296.18,398.59) and (297.58,400.05) .. (299.32,400.05) .. controls (301.05,400.05) and (302.46,398.59) .. (302.46,396.79) .. controls (302.46,394.99) and (301.05,393.53) .. (299.32,393.53) .. controls (297.58,393.53) and (296.18,394.99) .. (296.18,396.79) -- cycle ;
\draw  [color={rgb, 255:red, 74; green, 144; blue, 226 }  ,draw opacity=1 ][fill={rgb, 255:red, 74; green, 144; blue, 226 }  ,fill opacity=1 ] (497.45,396.79) .. controls (497.45,398.4) and (498.71,399.71) .. (500.26,399.71) .. controls (501.81,399.71) and (503.06,398.4) .. (503.06,396.79) .. controls (503.06,395.18) and (501.81,393.88) .. (500.26,393.88) .. controls (498.71,393.88) and (497.45,395.18) .. (497.45,396.79) -- cycle ;
\draw  [color={rgb, 255:red, 74; green, 144; blue, 226 }  ,draw opacity=1 ][fill={rgb, 255:red, 74; green, 144; blue, 226 }  ,fill opacity=1 ] (94.17,397.17) .. controls (94.17,398.78) and (95.43,400.09) .. (96.98,400.09) .. controls (98.53,400.09) and (99.79,398.78) .. (99.79,397.17) .. controls (99.79,395.56) and (98.53,394.26) .. (96.98,394.26) .. controls (95.43,394.26) and (94.17,395.56) .. (94.17,397.17) -- cycle ;
\draw  [color={rgb, 255:red, 74; green, 144; blue, 226 }  ,draw opacity=1 ][fill={rgb, 255:red, 74; green, 144; blue, 226 }  ,fill opacity=1 ] (125.79,368.19) .. controls (125.79,369.8) and (127.05,371.1) .. (128.6,371.1) .. controls (130.15,371.1) and (131.4,369.8) .. (131.4,368.19) .. controls (131.4,366.58) and (130.15,365.27) .. (128.6,365.27) .. controls (127.05,365.27) and (125.79,366.58) .. (125.79,368.19) -- cycle ;
\draw  [color={rgb, 255:red, 74; green, 144; blue, 226 }  ,draw opacity=1 ][fill={rgb, 255:red, 74; green, 144; blue, 226 }  ,fill opacity=1 ] (526.33,362.64) .. controls (526.33,364.25) and (527.58,365.55) .. (529.13,365.55) .. controls (530.68,365.55) and (531.94,364.25) .. (531.94,362.64) .. controls (531.94,361.03) and (530.68,359.72) .. (529.13,359.72) .. controls (527.58,359.72) and (526.33,361.03) .. (526.33,362.64) -- cycle ;
\draw  [color={rgb, 255:red, 126; green, 211; blue, 33 }  ,draw opacity=1 ][fill={rgb, 255:red, 126; green, 211; blue, 33 }  ,fill opacity=0.27 ][dash pattern={on 0.84pt off 2.51pt}] (509.08,362.64) .. controls (509.08,351.15) and (518.06,341.83) .. (529.13,341.83) .. controls (540.21,341.83) and (549.19,351.15) .. (549.19,362.64) .. controls (549.19,374.13) and (540.21,383.45) .. (529.13,383.45) .. controls (518.06,383.45) and (509.08,374.13) .. (509.08,362.64) -- cycle ;
\draw  [fill={rgb, 255:red, 126; green, 211; blue, 33 }  ,fill opacity=0.27 ] (95.91,447.26) -- (239.62,447.26) -- (239.62,596.38) -- (95.91,596.38) -- cycle ;
\draw   (298.25,446.84) -- (441.99,446.84) -- (441.99,596) -- (298.25,596) -- cycle ;
\draw   (499.19,446.84) -- (642.93,446.84) -- (642.93,596) -- (499.19,596) -- cycle ;
\draw  [color={rgb, 255:red, 126; green, 211; blue, 33 }  ,draw opacity=1 ][fill={rgb, 255:red, 126; green, 211; blue, 33 }  ,fill opacity=0.29 ] (307.45,447) -- (442.08,447) -- (442.08,588.29) -- (307.45,588.29) -- cycle ;
\draw  [color={rgb, 255:red, 74; green, 144; blue, 226 }  ,draw opacity=1 ][fill={rgb, 255:red, 74; green, 144; blue, 226 }  ,fill opacity=1 ] (326.32,563.51) .. controls (326.32,565.12) and (327.58,566.43) .. (329.13,566.43) .. controls (330.68,566.43) and (331.94,565.12) .. (331.94,563.51) .. controls (331.94,561.9) and (330.68,560.6) .. (329.13,560.6) .. controls (327.58,560.6) and (326.32,561.9) .. (326.32,563.51) -- cycle ;
\draw  [color={rgb, 255:red, 74; green, 144; blue, 226 }  ,draw opacity=1 ][fill={rgb, 255:red, 74; green, 144; blue, 226 }  ,fill opacity=1 ] (125.79,567.4) .. controls (125.79,568.47) and (126.63,569.34) .. (127.66,569.34) .. controls (128.69,569.34) and (129.53,568.47) .. (129.53,567.4) .. controls (129.53,566.32) and (128.69,565.45) .. (127.66,565.45) .. controls (126.63,565.45) and (125.79,566.32) .. (125.79,567.4) -- cycle ;
\draw  [color={rgb, 255:red, 74; green, 144; blue, 226 }  ,draw opacity=1 ][fill={rgb, 255:red, 74; green, 144; blue, 226 }  ,fill opacity=1 ] (525.26,561.85) .. controls (525.26,563.46) and (526.51,564.76) .. (528.06,564.76) .. controls (529.61,564.76) and (530.87,563.46) .. (530.87,561.85) .. controls (530.87,560.24) and (529.61,558.93) .. (528.06,558.93) .. controls (526.51,558.93) and (525.26,560.24) .. (525.26,561.85) -- cycle ;
\draw  [color={rgb, 255:red, 126; green, 211; blue, 33 }  ,draw opacity=1 ][fill={rgb, 255:red, 126; green, 211; blue, 33 }  ,fill opacity=0.27 ][dash pattern={on 0.84pt off 2.51pt}] (508.01,561.85) .. controls (508.01,550.36) and (516.99,541.04) .. (528.06,541.04) .. controls (539.14,541.04) and (548.12,550.36) .. (548.12,561.85) .. controls (548.12,573.34) and (539.14,582.66) .. (528.06,582.66) .. controls (516.99,582.66) and (508.01,573.34) .. (508.01,561.85) -- cycle ;
\draw [color={rgb, 255:red, 74; green, 144; blue, 226 }  ,draw opacity=1 ][fill={rgb, 255:red, 74; green, 144; blue, 226 }  ,fill opacity=1 ][line width=1.5]  [dash pattern={on 1.69pt off 2.76pt}]  (55.8,554.76) -- (136.02,638) ;
\draw [color={rgb, 255:red, 74; green, 144; blue, 226 }  ,draw opacity=1 ][fill={rgb, 255:red, 74; green, 144; blue, 226 }  ,fill opacity=1 ][line width=1.5]  [dash pattern={on 1.69pt off 2.76pt}]  (257.98,554.38) -- (338.2,637.62) ;
\draw [color={rgb, 255:red, 74; green, 144; blue, 226 }  ,draw opacity=1 ][fill={rgb, 255:red, 74; green, 144; blue, 226 }  ,fill opacity=1 ][line width=1.5]  [dash pattern={on 1.69pt off 2.76pt}]  (459.08,554.38) -- (539.29,637.62) ;
\draw  [color={rgb, 255:red, 0; green, 0; blue, 0 }  ,draw opacity=1 ][fill={rgb, 255:red, 255; green, 255; blue, 255 }  ,fill opacity=1 ][dash pattern={on 0.84pt off 2.51pt}] (95.46,519.62) -- (126.01,566.51) -- (95.62,548.54) -- cycle ;
\draw  [color={rgb, 255:red, 0; green, 0; blue, 0 }  ,draw opacity=1 ][fill={rgb, 255:red, 255; green, 255; blue, 255 }  ,fill opacity=1 ][dash pattern={on 0.84pt off 2.51pt}] (141.35,596.47) -- (172.47,596.14) -- (128.92,569.32) -- cycle ;
\draw  [color={rgb, 255:red, 0; green, 0; blue, 0 }  ,draw opacity=1 ][fill={rgb, 255:red, 255; green, 255; blue, 255 }  ,fill opacity=1 ][dash pattern={on 0.84pt off 2.51pt}] (342.98,588.42) -- (369.92,588.54) -- (331.04,565.07) -- cycle ;
\draw  [color={rgb, 255:red, 0; green, 0; blue, 0 }  ,draw opacity=1 ][fill={rgb, 255:red, 255; green, 255; blue, 255 }  ,fill opacity=1 ][dash pattern={on 0.84pt off 2.51pt}] (307,524.54) -- (326.94,561.15) -- (307.15,551.46) -- cycle ;
\draw  [color={rgb, 255:red, 0; green, 0; blue, 0 }  ,draw opacity=1 ][fill={rgb, 255:red, 255; green, 255; blue, 255 }  ,fill opacity=1 ][dash pattern={on 0.84pt off 2.51pt}] (519.17,543.15) -- (525.59,559.66) -- (509.35,555.02) -- (509.95,552.8) -- (510.95,550.99) -- (512.56,548.63) -- (514.83,546.21) -- (516.97,544.47) -- (518.57,543.36) -- cycle ;
\draw  [color={rgb, 255:red, 0; green, 0; blue, 0 }  ,draw opacity=1 ][fill={rgb, 255:red, 255; green, 255; blue, 255 }  ,fill opacity=1 ][dash pattern={on 0.84pt off 2.51pt}] (545.04,572.77) -- (545.91,571.11) -- (546.65,569.1) -- (530.47,564.03) -- (538.29,579.57) -- (539.9,578.6) -- (541.43,577.28) -- (543.04,575.75) -- (544.17,574.23) -- cycle ;
\draw  [draw opacity=0] (115,554.5) .. controls (114.8,554.31) and (114.63,554.07) .. (114.52,553.79) .. controls (114.07,552.67) and (114.75,551.34) .. (116.03,550.83) .. controls (117.31,550.32) and (118.71,550.82) .. (119.16,551.95) .. controls (119.27,552.22) and (119.31,552.51) .. (119.29,552.8) -- (116.84,552.87) -- cycle ; \draw   (115,554.5) .. controls (114.8,554.31) and (114.63,554.07) .. (114.52,553.79) .. controls (114.07,552.67) and (114.75,551.34) .. (116.03,550.83) .. controls (117.31,550.32) and (118.71,550.82) .. (119.16,551.95) .. controls (119.27,552.22) and (119.31,552.51) .. (119.29,552.8) ;  
\draw  [dash pattern={on 0.75pt off 0.75pt on 0.75pt off 0.75pt}]  (299,41.63) -- (307.25,41.63) ;
\draw    (273.75,36.38) .. controls (279.54,17.56) and (294.11,21.93) .. (301.93,36.27) ;
\draw [shift={(302.75,37.88)}, rotate = 244.12] [color={rgb, 255:red, 0; green, 0; blue, 0 }  ][line width=0.75]    (4.37,-1.32) .. controls (2.78,-0.56) and (1.32,-0.12) .. (0,0) .. controls (1.32,0.12) and (2.78,0.56) .. (4.37,1.32)   ;
\draw [color={rgb, 255:red, 208; green, 2; blue, 27 }  ,draw opacity=1 ][line width=0.75]    (96.13,535.44) -- (108.82,548.43) -- (155.63,596.31) ;
\draw  [dash pattern={on 0.75pt off 0.75pt on 0.75pt off 0.75pt}]  (299.17,260.29) -- (307.42,260.29) ;
\draw  [dash pattern={on 0.84pt off 2.51pt}]  (306.69,30.85) -- (307.46,179.77) ;
\draw  [dash pattern={on 0.84pt off 2.51pt}]  (440.8,172.5) -- (298.4,171.7) ;
\draw  [dash pattern={on 0.75pt off 0.75pt on 0.75pt off 0.75pt}]  (298.72,460.33) -- (306.97,460.33) ;
\draw [color={rgb, 255:red, 208; green, 2; blue, 27 }  ,draw opacity=1 ][fill={rgb, 255:red, 74; green, 144; blue, 226 }  ,fill opacity=1 ][line width=0.75]    (307.4,540.25) -- (353,588.08) ;
\draw [color={rgb, 255:red, 208; green, 2; blue, 27 }  ,draw opacity=1 ][fill={rgb, 255:red, 74; green, 144; blue, 226 }  ,fill opacity=1 ][line width=0.75]    (513.47,548.06) -- (543.59,575.47) ;

\draw (152.48,149.31) node [anchor=north west][inner sep=0.75pt]    {$\mathcal{S}$};
\draw (57.64,192.87) node [anchor=north west][inner sep=0.75pt]    {$\ker(\mathbf{A})$};
\draw (149.76,4.71) node [anchor=north west][inner sep=0.75pt]  [rotate=-359]  {$\varphi _{\mathrm{sB}}$};
\draw (357.84,3.48) node [anchor=north west][inner sep=0.75pt]    {$\varphi _{\mathrm{B}}$};
\draw (560.79,4.14) node [anchor=north west][inner sep=0.75pt]    {$\varphi _{\ell _{2}}$};
\draw (2.95,97.2) node [anchor=north west][inner sep=0.75pt]    {$S_{\ker}^{\theta } =\emptyset $};
\draw (509.81,90.35) node [anchor=north west][inner sep=0.75pt]    {$\mathcal{B} \supseteq \mathcal{S}$};
\draw (138.84,354.9) node [anchor=north west][inner sep=0.75pt]    {$\mathcal{S}$};
\draw (49.62,404.84) node [anchor=north west][inner sep=0.75pt]    {$\ker(\mathbf{A})$};
\draw (4.02,314.44) node [anchor=north west][inner sep=0.75pt]    {$S_{\ker}^{\theta } =\emptyset $};
\draw (137.77,554.11) node [anchor=north west][inner sep=0.75pt]    {$\mathcal{S}$};
\draw (48.4,604.05) node [anchor=north west][inner sep=0.75pt]    {$\ker(\mathbf{A})$};
\draw (2.95,513.65) node [anchor=north west][inner sep=0.75pt]    {$S_{\ker}^{\theta } \neq \emptyset $};
\draw (118.22,533.2) node [anchor=north west][inner sep=0.75pt]    {$\theta $};
\draw (265.67,33.95) node [anchor=north west][inner sep=0.75pt]    {$ \begin{array}{l}
\varepsilon \\
\end{array}$};

\end{tikzpicture}
    
    \caption{Validity regions in which~\cref{eq:rrsc} holds depending on 1) the rank of the matrix $\bA$, and 2) the choice of $\varphi$. In each plot, the green shaded areas represent regions in which~\cref{eq:rrsc} holds (for some fixed $\epsilon >0$ and $\theta \in (0,\pi/2)$), and the red lines in the third row indicate that~\cref{eq:rrsc} is not satisfied therein (cf.~\Cref{lem:necessary condition}). Note that $\varepsilon$ and $\theta$ can be chosen arbitrarily close to $0$ at the cost of worse constants $\mu$ and $\alpha_\cS$.}
    \label{fig:valid_regions}
\end{figure}

\begin{proposition}[Restricted relative strong convexity]\label{prop:rrsc for KL}
Let $F$ be the KL divergence  defined in~\cref{eq:fid_KL}. Then, we have that
\begin{itemize}
\item  For any $\theta\in(0,\pi/2)$, $F$ is restricted strongly convex relative to $\varphi_{\mathrm{sB}}$ over $\cC \setminus \cS^{\theta}_{\ker}$.
\item For any $\varepsilon>0$ and any $\theta\in(0,\pi/2)$, $F$ is restricted strongly convex relative to  $\varphi_\mathrm{B}$ over 
\begin{equation} \label{eq:linconv_set_for_Burg}
\cK_\varepsilon:= \left\lbrace  \bx \in \R^N_{\geq \varepsilon} \setminus \cS^{\theta}_{\ker} \;\mid\; \bar \bx_{\ell_2} \in \R^N_{\geq \varepsilon}\right\rbrace.
\end{equation} 
\item For any $\theta\in(0,\pi/2)$ and any compact set $\cB \subseteq \cC $ containing $\cS$, $F$ is restricted strongly convex (relative to $\varphi_\mathrm{\ell_2}$) over  $\cB \setminus \cS^{\theta}_{\ker}$. 
\end{itemize}
\end{proposition}

We defer its proof to~\Cref{app:rrsc}. The validity regions where \eqref{eq:rrsc} holds, as established in \Cref{prop:rrsc for KL}, are illustrated in dimension two for different scenarios in~\Cref{fig:valid_regions}. We next provide an analogous result for the restricted symmetry coefficient. 

\begin{proposition}[Restricted symmetry coefficient]\label{prop:restr_alpha_KL}
Let $\cS$ be the solution set of~\cref{eq:comp_obj}. We have that 
\begin{itemize}
\item $\alpha_{\cS}(\varphi_{\mathrm{sB}})>0$ over $\cC$,
\item $\alpha_{\cS \cap \R_{\geq \varepsilon}^N}(\varphi_\mathrm{B})>0$ over $\R_{\geq \varepsilon}^N$,
\item $\alpha(\varphi_{\ell_2})=1$ over $\R^N$.
\end{itemize}
\end{proposition}
In the above proposition, ``over $\cB$'' for $\cB \subseteq \R^N$ should be understood as the definition of $\alpha_\cS(\varphi)$  in~\cref{eq:alpha(h)} with the constraint $\bx \in \mathrm{int}(\dom \varphi)$ replaced by $\bx \in \cB$. We defer its proof to~\Cref{sec:alphaKL_is_positive}. While the above result for $\varphi_{\ell_2}$ is trivial, the remaining cases are novel to this work. The only known result was the negative one that $\varphi_\mathrm{B}$ suffers from a total lack of symmetry over its domain (i.e., $\alpha(\varphi_\mathrm{B}) = 0$)~\cite{bauschke2017descent}. By restricting the computation of this symmetry coefficient to $\R_{\geq \varepsilon}^N$ (for $\bx$) and the compact set $\cS$ (for $\by$), we obtain a positive result. Finally, we describe the different relative $L$-smooth constants for our choices of $\varphi$.

\begin{proposition}[Relative smoothness]
\label{prop:rel_smooth_KL}
Let $F$ be the KL divergence  defined in~\cref{eq:fid_KL}. Then, we have that
\begin{itemize}
\item $F$ is $L$-smooth relative to $\varphi_{\mathrm{sB}}$ for any $L\geq\sum_m \omega_m^2 y_m$, with $\omega_m:=\max\{1,[\bA\boldsymbol{1}]_m \xi/b_m\}$, $m=1,\ldots, M$.
\item $F$ is $L$-smooth relative to $\varphi_\mathrm{B}$ for any $L\geq \sum_m y_m$.
\item $F$ is $L$-smooth (relative to $\varphi_{\ell_2}$) for any $L\geq \|\by/\mathbf{b}^2\|_\infty \|\mathbf{A}\|_2^2$.
\end{itemize}
\end{proposition}
We defer the proof of the result to~\Cref{proof:rel_smooth_KL}. While the case $\varphi_{\mathrm{B}}$ is a direct extension of~\cite[Lemma 7]{bauschke2017descent}, the proof for $\varphi_{\mathrm{sB}}$ is slightly more involved. Finally, for $\varphi_{\ell_2}$, $L$ simply reduces to the Lipschitz gradient constant of $F$ in \eqref{eq:fid_KL}, given in~\cite[Lemma 1]{harmany2011spiral}.

\begin{table}[t]
\centering
\caption{Summary of linear convergence results (both theoretical and experimental) for different KL regression settings, depending on the rank of $\bA$ and the presence or absence of a regularizer. Note that, in the presence of a regularizer, we only discuss cases where the solution is unique, which is one of the typical roles of adding a regularizer. The notation $\cF_{\bA}^+(\bz) := (\bz + \ker(\bA)) \cap \cC$ denotes the nonnegative fiber of $\bA$ through~$\bz$, while $\Gamma_{\bA, \theta}^+(\bz)$ describes the nonnegative (double) cone with apex 
$\bz$, ``axis'' $\ker(\bA)$, and angle $\theta$. For theoretical linear convergence, \cmarkm~indicates that linear convergence is possible but may depend on the trajectory of iterates, while an absence of mark means that the setting is not covered by the theory. For experimental linear convergence, \cmarkm~indicates that both linear and sublinear convergence were observed, depending on the choice of the initial point, whereas \cmark~(resp., \xmark) indicates that linear convergence was always (resp., never) observed. Finally, the column ``Fig.'' refers to the corresponding figure in our experiments. We do not report here the case $\varphi_{\ell_2}$, as it only presents a local behavior, but integrate it in each figure.}
\label{tab:comparison}
\begin{tabular}{cclcccccc}
\toprule 
\toprule
&  \multirow{2}{*}{$\mathrm{rank}(\bA)$}& \multicolumn{1}{c}{\multirow{2}{*}{$\cS$}} & \multirow{2}{*}{$\cS_{\ker}^{\theta}$} & \multicolumn{2}{c}{Lin. cv. (th)} &  \multicolumn{2}{c}{Lin. cv. (exp)} & \multirow{2}{*}{Fig.}\\
& & & & $\varphi_{\mathrm{sB}}$ & $\varphi_{\mathrm{B}}$  & $\varphi_{\mathrm{sB}}$ & $\varphi_{\mathrm{B}}$ \\
\midrule
\multirow{3}{*}{\rotatebox[origin=c]{90}{$R = 0$}} &
 $< N$  & $\cF_{\bA}^+(\bar \bx)$ for $\bar \bx$ sol. & $\emptyset$ & \cmark & \cmarkm & \cmark & \cmarkm & \ref{fig:nfr_no_reg} \\
& $=N$      & $\{\bar \bx \} \subset \mathrm{int}(\cC) $  & $\emptyset$ & \cmark & \cmark & \cmark & \cmark & \ref{fig:fr_no_reg_int}  \\
& $=N$   & $\{\bar \bx \} \subset \partial \cC  $  & $\emptyset$  &  \cmark & & \cmark & \xmark  & \ref{fig:fr_no_reg_boundary} \\
\midrule
\multirow{4}{*}{\rotatebox[origin=c]{90}{$R \neq 0$}} &  $< N$  & $\{\bar \bx \} \subset \mathrm{int}(\cC ) $ &  $\Gamma_{\bA,\theta}^+(\bar \bx)  \setminus \{\bar \bx\}$  & \cmarkm & \cmarkm & \multicolumn{3}{c}{(see~\Cref{rem:NFR_reg_solIntC})}\\
& $<N$      & $\{\bar \bx \} \subset \partial \cC $ &  $\Gamma_{\bA,\theta}^+(\bar \bx) \setminus \{\bar \bx\}$ & \cmarkm & & \cmark & \xmark & \ref{fig:NFR reg}\\
& $=N$      & $\{\bar \bx \} \subset \mathrm{int}(\cC ) $ & $\emptyset$ & \cmark & \cmark & \cmark & \cmark & \ref{fig:fr_reg_int}\\
& $=N$   & $\{\bar \bx \} \subset \partial \cC  $  & $\emptyset$  &  \cmark  & & \cmark & \xmark & \ref{fig:fr_reg_boundary}\\
\bottomrule
\bottomrule
\end{tabular}
\end{table}

From the above three propositions, one can see that the linear convergence rate of~\eqref{eq:BPGM} for the considered KL regression problems depends on both $\bA$ and $\cS$. In \Cref{tab:comparison} (column ``Lin. cv. (th)''), we summarize (for $\varphi_\mathrm{sB}$ and $\varphi_{\mathrm{B}}$) the different situations and indicate whether linear convergence is either fully, or only partially, guaranteed by our theory, as this may depend on the trajectory of iterates. Indeed, as illustrated in Figure~\ref{fig:valid_regions}, the condition~\eqref{eq:rrsc} is not necessarily satisfied over the entire set $\cC$.  In such cases, linear convergence of \eqref{eq:BPGM} can only be expected if the iterates remain in the validity region. While for $\varphi_\mathrm{sB}$ linear convergence is systematically guaranteed for all cases where $\cS^{\theta}_{\ker} = \emptyset$, for $\varphi_\mathrm{B}$ the additional requirement that the solution $\bar \bx$, or the convergence point $\bx^\infty$ when multiple solutions exist, belongs to $\mathrm{int}(\cC)$ is necessary. Indeed, if the solution belongs to $\partial \cC$, there does not exist $\varepsilon >0$ such that it also belongs to $\mathcal{K}_\varepsilon$ in~\Cref{prop:rrsc for KL} and to $\R_{\geq \varepsilon}^N$ in \Cref{prop:restr_alpha_KL}. Similarly, if all iterates $\bx^k$ (or at least iterates after $K$ iterations) lie in the set $(\{\bar \bx\} + \ker(\bA))\cap \cC$, there does not exist $\theta \in (0,\pi/2)$ such that these iterates lie outside $\cS^{\theta}_{\ker}$, and we cannot guarantee linear convergence. However, this constitutes a very edge case that we did not manage to observe in our $\ell_1$-regularized numerical experiments in \Cref{sec:exp_comment}.

\begin{remark}[On the case $\mathbf{b} = \mathbf{0}$]
    Observing that, for any $\mathbf{b} > \mathbf{0}$, the following inequality holds:
$$
\sum_{m=1}^{M} \frac{y_{m}[\bA(\bx - \bar\bx_{\ell_2})]_m^2}{[\bA\bx]_m [\bA\bar\bx_{\ell_2}]_m}
\geq
\sum_{m=1}^{M} \frac{y_{m}[\bA(\bx - \bar\bx_{\ell_2})]_m^2}{([\bA\bx]_m + b_m)([\bA\bar\bx_{\ell_2}]_m + b_m)},
$$
we deduce that the results stated in \Cref{prop:rrsc for KL} (existence of a positive constant for~\eqref{eq:rrsc}) also hold for $\mathrm{KL}(\by, \bA \cdot)$ (unsmoothed). Given that the bound $L \geq \sum_m y_m$ for $\varphi_{\mathrm{B}}$ in \Cref{prop:rel_smooth_KL} remains valid when $\mathbf{b} = \mathbf{0}$, our analysis for $\varphi_{\mathrm{B}}$ directly applies to this unsmoothed case. Regarding $\varphi_{\mathrm{sB}}$, we observe that the bound on $L$ in \Cref{prop:rel_smooth_KL} tends to $+\infty$ as $\mathbf{b}$ tends to $\mathbf{0}$. Therefore, a dedicated analysis would be required to finely analyze this case.
\end{remark}

\subsection{Numerical experiments}\label{sec:experiments}
This section presents numerical experiments designed to to illustrate our theoretical findings regarding the convergence of~\eqref{eq:BPGM} with backtracking for Problem \eqref{eq:comp_obj} with the KL objective defined in~\eqref{eq:fid_KL} (including a possible $\ell_1$ regularizer). In this setting the~\eqref{eq:BPGM} iteration writes:
\begin{equation} \label{eq: BPGM iter implicit}
\bx^{k+1} = \argmin_{\bx \in \cC} \left<\mathbf{v}^k-\frac{1}{\tau^k}\nabla\varphi(\bx^k), \bx\right> + \frac{1}{\tau^k} \varphi(\bx) +\lambda\|\bx\|_1,
\end{equation}
where $\lambda\geq 0$ is a regularization parameter ($\lambda=0$ indicates  unregularized KL regression), $\tau^k>0$ is chosen according to the backtracking step-size rule provided in~\Cref{coro:lin_cv_backtrack} and, finally, $\mathbf{v}^k:=\bA^T\nabla \mathrm{KL}(\by,\bA\bx^k+\mathbf{b})$ stands for the gradient of $F$ in~\cref{eq:comp_obj}, with 
\begin{equation}
    v_n^k := \left\langle \mathbf{a}_n, \mathbf{1} - \frac{\mathbf{y}}{\mathbf{A}\mathbf{x}^k + \mathbf{b}} \right\rangle \label{eq: gradient of KL}
\end{equation}
for all $n=1,\ldots, N$.

\begin{remark}[Implicit regularization of unregularized KL regression]\label{rem:implicit_reg}
     When $M\ll N$ and $\lambda=0$, it has been shown in~\cite[Proposition 1]{Boyer2019} that, in this context, there is a solution to~\cref{eq:comp_obj} with $\cC = \R_{\geq 0}^N$ which is $M-$sparse; i.e., which has at most $M$ nonzero elements. This indicates that, in some sense, sparsity is induced from the non-negativity constraint itself, see also~\cite{Donoho2005}.
\end{remark}

We now provide an explicit expression of~\cref{eq: BPGM iter implicit} for each of the three different choices of $\varphi$ we consider (i.e., $\varphi_\mathrm{sB}$, $\varphi_\mathrm{B}$, and $\varphi_{\ell_2}$). Note that, a priori,~\cref{eq: BPGM iter implicit} requires to solve an inner optimization problem at every iteration, but for these three choices of $\varphi$ we have access to a closed-form solution. While for $\varphi_{\mathrm{sB}}$ the componentwise iteration of~\cref{eq: BPGM iter implicit} writes
\begin{equation}\label{iter with sBurg}
x^{k+1}_n=\max\left\{\frac{x_n^k+\xi}{1+\tau^k (v_n^k+\lambda)(x_n^k+\xi)}-\xi, 0\right\},
\end{equation}
the choice $\xi=0$ leads to the componentwise iteration of~\cref{eq: BPGM iter implicit} with $\varphi_{\mathrm{B}}$: 
\begin{equation}\label{iter with Burg}
x^{k+1}_n=\frac{x_n^k}{1+\tau^k (v_n^k+\lambda)x_n^k}.
\end{equation}
Notice that, opposed to the choice $\varphi_{\mathrm{sB}}$, the $\max\{\cdot, 0\}$ is not needed here as, for $\bx^0\in\R_{>0}^N$, all subsequent iterates $\bx^k$ remain in $\R_{>0}^N$, both with $\tau^k=\tau\leq 1/L$ 
or $\tau^k$ chosen according to the backtracking rule of~\Cref{coro:lin_cv_backtrack}. Finally, for $\varphi_{\ell_2}$ we have that
\begin{equation}\label{eq:projgd_reg}
        x^{k+1}_n = \max \left\{\mathrm{S}_{\tau^k \lambda} (x^{k}_n - \tau^k v_n^k), \, 0 \right\},
\end{equation}
where $\mathrm{S}_{\tau \lambda}(z):= \text{sign}(z) \max(|z| - \tau \lambda, 0)$ is the well-know Soft-Tresholding operator \cite{daubechies2004iterative}. Notice that, if $\lambda=0$, then clearly $\mathrm{S}_{0}=\mathrm{Id}$ and~\cref{eq:projgd_reg} simply becomes the projected gradient descent method. 

Finally, in each of the experiments we will include a comparison with the classical Richardson--Lucy (RL) algorithm~\cite{FI2011,lucy74,richardson1972bayesian}, which can be derived either as an Expectation-Maximization algorithm under a composite Poisson model or as a majorization-minimization algorithm using Jensen's inequality. The componentwise iteration is multiplicative, preserving the non-negativity at each iteration, and writes
\begin{equation}\label{eq:Rich_Lucy}
x_n^{k+1} =  \frac{x_n^k}{\ba_n^T \mathbf{1} + \lambda}\left\langle \mathbf{a}_n, \frac{\mathbf{y}}{\mathbf{A}\mathbf{x}^k + \mathbf{b}} \right\rangle  = \frac{x_n^k}{\ba_n^T \mathbf{1} + \lambda} (\ba_n^T \mathbf{1} - v_n^k)
\end{equation}
for all $n=1,\ldots, N$.

In the following, we provide further details on how we generate the data, we discuss the effect of the smoothness parameter $\xi>0$ in $\varphi_{\mathrm{sB}}$ on the convergence of~\cref{iter with sBurg} and, finally, we will provide numerical evidence to each of the scenarios described in \Cref{tab:comparison} by varying over 1) the rank of the matrix $\bA$, 2) the presence (or not) of the $\ell_1$-regularizer, and 3) the trajectory of iterates through the selection of specific initial points. \\

\paragraph{\textbf{Data Generation}} \label{sec: data gen}
To evaluate our approach, we consider a synthetic setup within the context of Poisson inverse problems. We set the signal dimension to $N=500$ and the number of measurements to $M \leq N$. The data $\by \in \mathbb{R}^M$ is generated according to the following model:
\begin{equation}\label{eq:PoissModel numerics}
    \by \sim \mathcal{P}(\by_0)/Q, \quad \text{with } \by_0 = Q\bA \bx^* + \mathbf{b}, 
\end{equation}
where $\mathbf{b} > \boldsymbol{0}$ is a background signal with $b_m = 10^{-3}$ for all $m=1,\ldots,M$, and $Q>0$ a gain factor that we fix to $Q=100$.  The ground-truth vector $\bx^* \in \mathbb{R}^N$ is drawn from a uniform distribution, such that $x_n^* \sim \cU(0, 1)$ for all $n=1,\ldots,N$. The matrix $\bA$ is constructed by drawing its entries from a uniform distribution, i.e.  $a_{m,n} \sim \cU(0, 1)$ and then rescaling its singular values to get a better condition number. Since the constants in both the sublinear and linear convergence rates depend on the condition number of $\bA$, this normalization allows us to illustrate these behaviors in fewer iterations, always reaching machine precision when linear convergence holds. Further details on the procedure can be found in \Cref{app: construction of A}.

For each experiment, we perform $30$ independent trials using different random initializations $\bx^0 \in \cC \cap \mathrm{int}(\dom\varphi)$, where $x_n^0 \sim \cU(10^{-3}, 10^4)$ for all $n=1,\ldots,N$. In each of the different convergence plots, solid lines represent the median behavior across these $30$ trials, while the shaded regions indicate the $1^{\mathrm{st}}$ and $99^{\mathrm{th}}$ percentiles. Finally, the optimal functional value $\bar J$ is set as the minimum value achieved among all evaluated algorithms at the final iteration.\\

\paragraph{\textbf{Code details:}} All simulations were implemented in Python and run on a MacBook Pro equipped with an Apple M4 Pro chip and 24 GB of RAM The code is publicly available at \url{https://github.com/christiandaniele/Linear_convergence_BPGM}. 

\begin{figure}[t]
    \centering
    \includegraphics[width=0.8\linewidth]{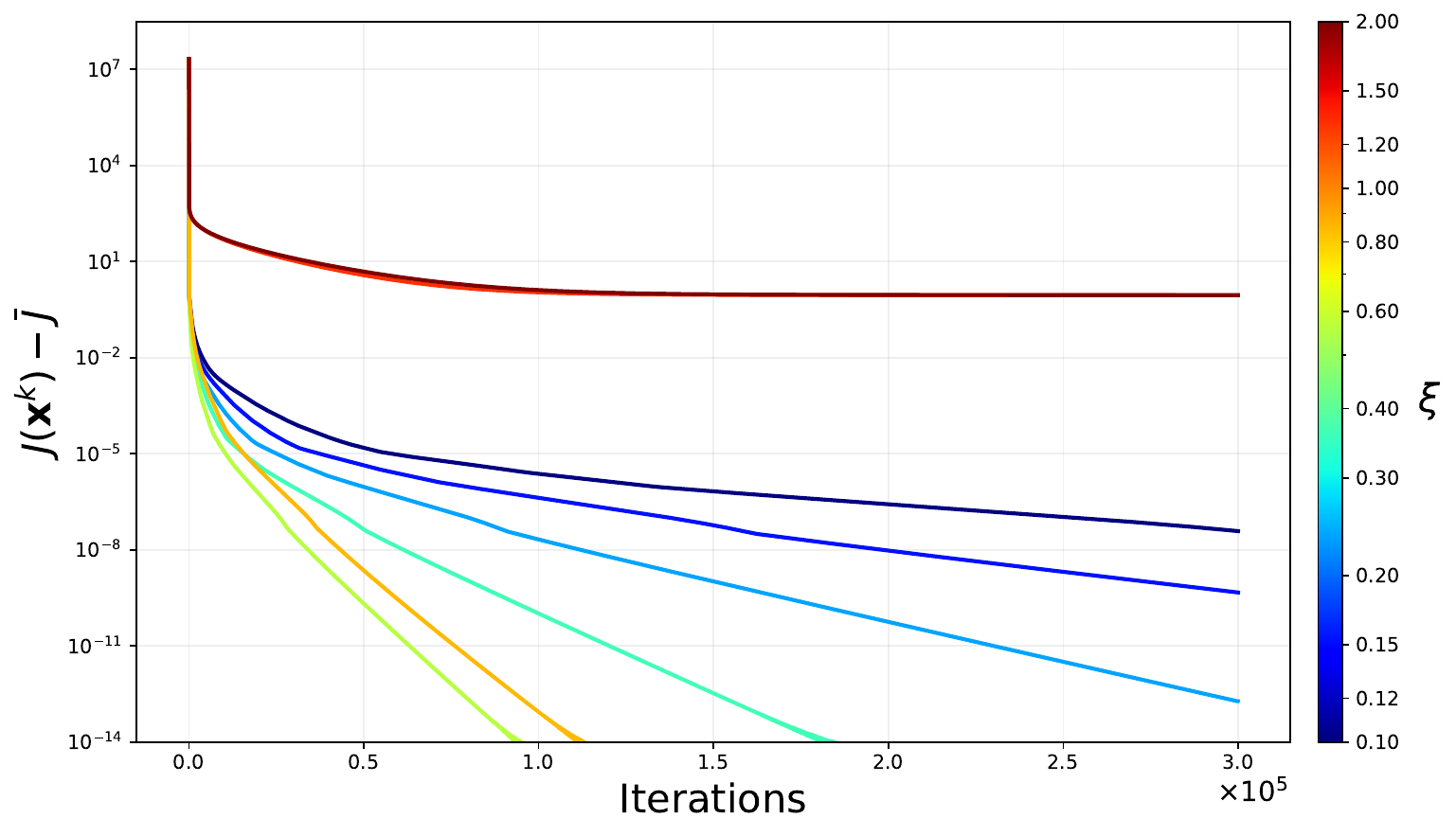}
    \caption{Convergence plots of~\Cref{eq:BPGM} with $\varphi_{\mathrm{sB}}$ for different values of $\xi$. Observe that the linear convergence gets faster when increasing $\xi$  until the optimal value $\xi\sim0.6$ in this case. After, the convergence of~\Cref{eq:BPGM} degrades when continuing to increase $\xi$.
    }
    \label{fig:comparison_xi}
\end{figure}

\subsubsection{\texorpdfstring{Effect of the smoothness parameter in $\varphi_{\mathrm{sB}}$}{Effect on the smoothing parameter in varphisB}}  In~\Cref{fig:comparison_xi}, we evaluate the impact of the smoothing parameter $\xi>0$ on the convergence of~\cref{eq:BPGM} in, for simplicity, the unregularized full-rank setting ($M=N$, $\lambda=0$). The results illustrate a clear trade-off, as extreme values of $\xi$ may degrade the convergence speed. This behavior perfectly aligns with our theoretical findings. While a larger $\xi$ favorably increases the restricted relative strong convexity constant $\mu$, we see that $\mu$ cannot grow arbitrarily: from the proof of~\Cref{prop:rrsc for KL} for $\varphi_{\mathrm{sB}}$, we have that
$$
\mu\geq \min\left\{c_1\xi^2, \frac{c_2\xi}{c_3+\xi}\right\},
$$
for some constants $c_1, c_2,c_3>0$. Looking at the right hand side above as a function of $\xi$, $g_\mu(\xi)$, we observe that both of the components inside the minimum are strictly increasing and, thus, $\sup_{\xi > 0} g_\mu(\xi) = \lim_{\xi \to \infty} g(\xi) = c_2$. In this way, we have shown that the lower bound of $\mu$ is globally bounded. However, from~\Cref{prop:rel_smooth_KL} for $\varphi_{\mathrm{sB}}$ we see that
$$
L\geq \max\{c_4,c_5\xi^2\},
$$
for some $c_4, c_5>0$. As for $\mu$, the right hand side above as a function of $\xi$, $g_L(\xi)$, is strictly increasing, but $\sup_{\xi>0}g_L(\xi)=+\infty$. Hence, $\xi$ increases the relative smoothness constant $L$ to $+\infty$, which forces a smaller stepsize and worsens the linear convergence rate. It is therefore expected to observe that, after a while, increasing $\xi$ does not bring anything for $\mu$ and continues degrading $L$. To balance these competing effects, we systematically fine-tune $\xi$, among $30$ different values, to achieve optimal performance in all subsequent experiments. We finally point out that, although the theoretical rate of convergence does also depend on the symmetry coefficient $\alpha_\cS(\varphi)$ which may itself vary with $\xi$, it seems, from our experiments that it does not impact the general behavior we described from the analysis of $\mu$ and $L$.


\subsubsection{Adequacy between theoretical findings and numerical experiments.}\label{sec:exp_comment} 
We next demonstrate the close alignment between our theoretical findings and the numerical results. Overall, our empirical setup successfully validates the predicted convergence behaviors detailed in~\Cref{tab:comparison}, column ``Lin. cv. (exp)''.

\begin{figure}[t]
    \centering
    \begin{subfigure}[b]{0.48\textwidth}
        \centering
        \includegraphics[width=\textwidth]{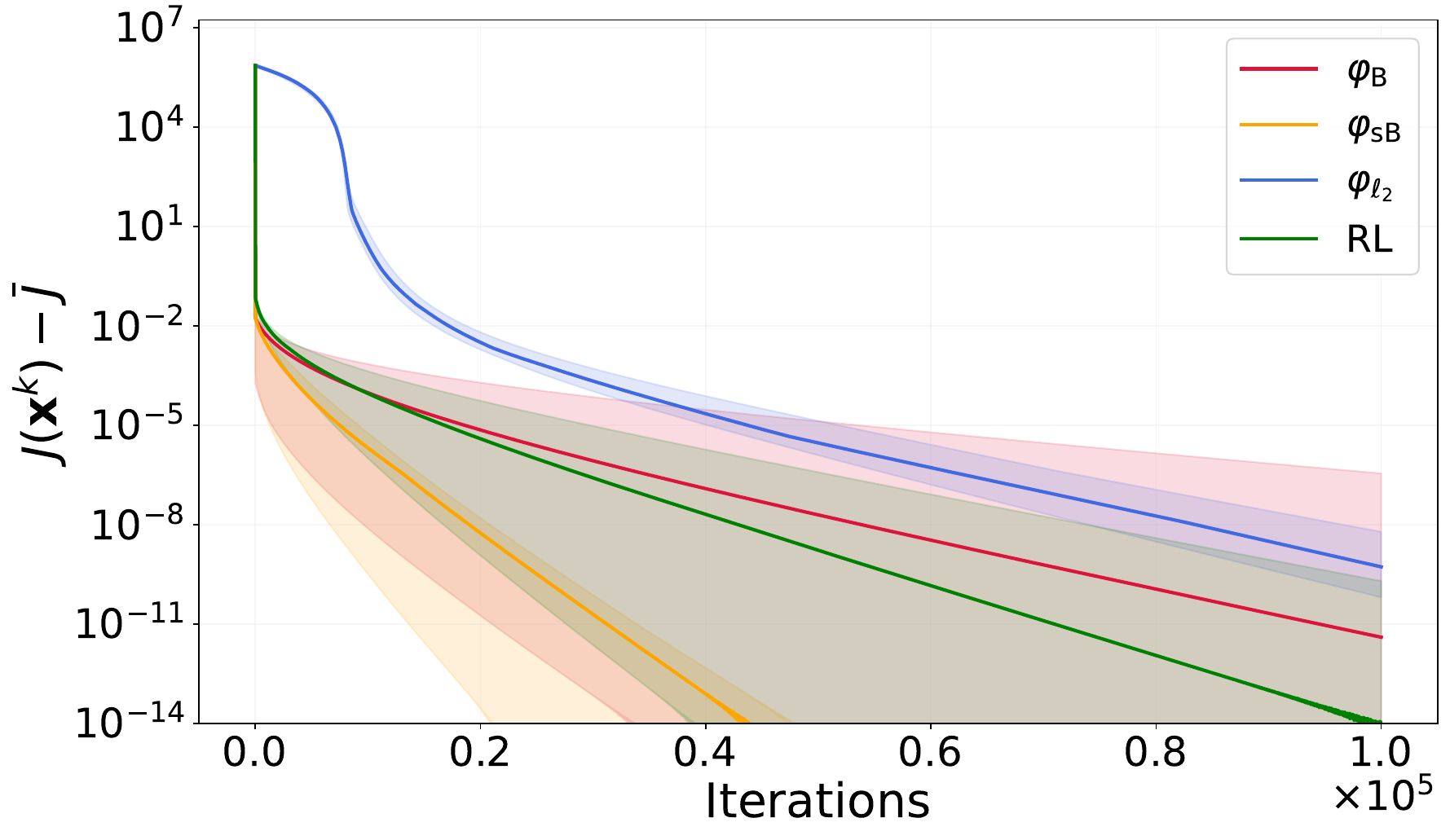} 
        \caption{Different choices of $\varphi$ and RL.}
        \label{fig:nfr_no_reg_int}
    \end{subfigure}
    \hfill 
    \begin{subfigure}[b]{0.48\textwidth}
        \centering
        \includegraphics[width=\textwidth]{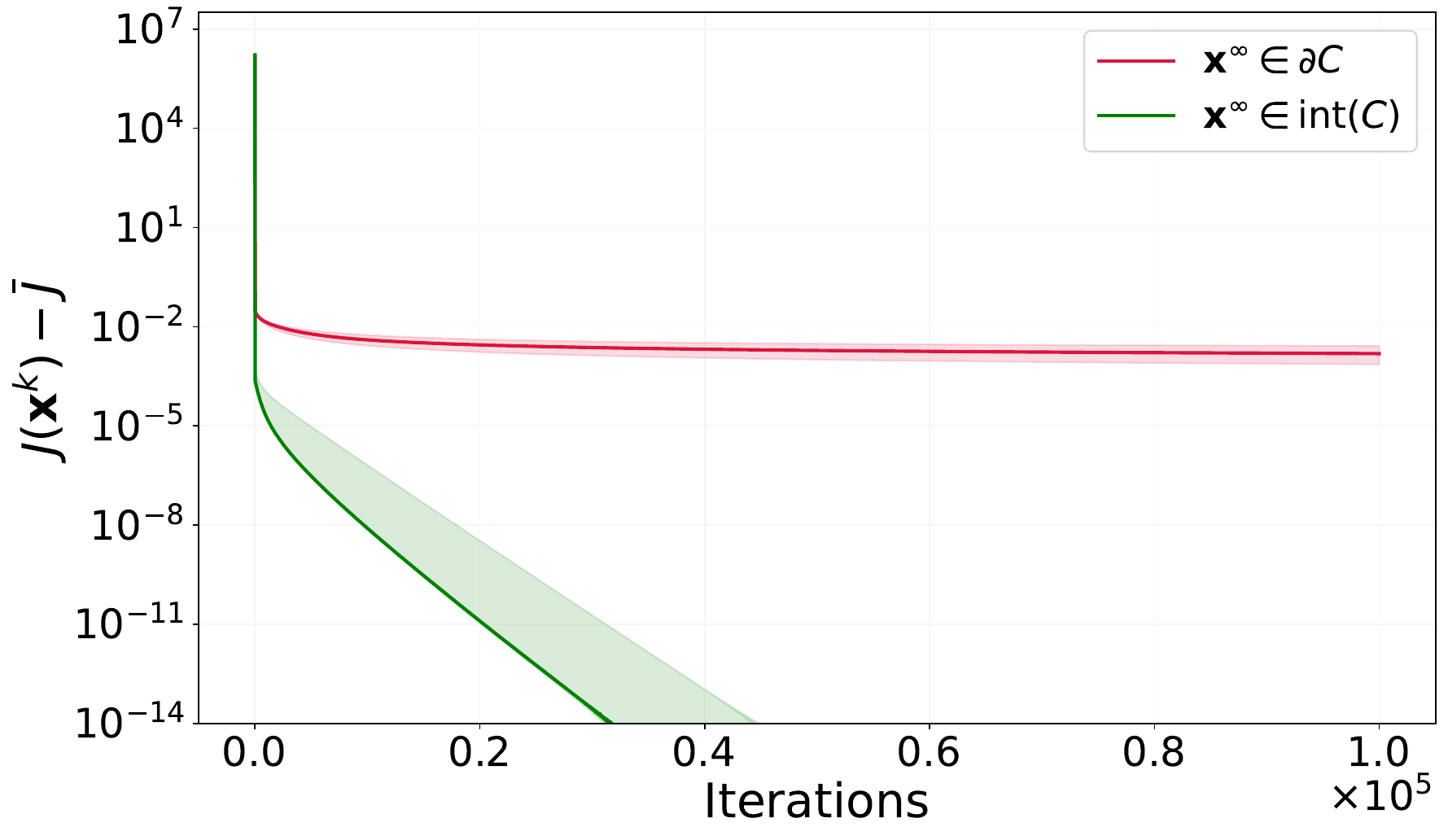} 
        \caption{Focus on $\varphi_{\mathrm{B}}$.}
        \label{fig:nfr_no_reg_boundary}
    \end{subfigure}
    
    \caption{On the left, we plot the convergence of~\cref{eq:BPGM}   with backtracking and the three choices of $\varphi$, as well as the RL algorithm, for unregularized KL in the case $N=500$, $M=400$. On the right, we focus on $\varphi_{\mathrm{B}}$, where we separate the cases (among the 30 initial points considered to make the left graph) in which $\bx^\infty \in  \partial\cC$ or $\bx^\infty \in  \mathrm{int}(\cC)$: while only sublinear convergence is observed when $\bx^\infty \in \partial\cC$, we always observe linear convergence when $\bx^\infty \in  \mathrm{int}(\cC)$.}
    \label{fig:nfr_no_reg}
\end{figure}

\begin{figure}[t]
    \centering
    \begin{subfigure}[b]{0.48\textwidth}
        \centering
        \includegraphics[width=\textwidth]{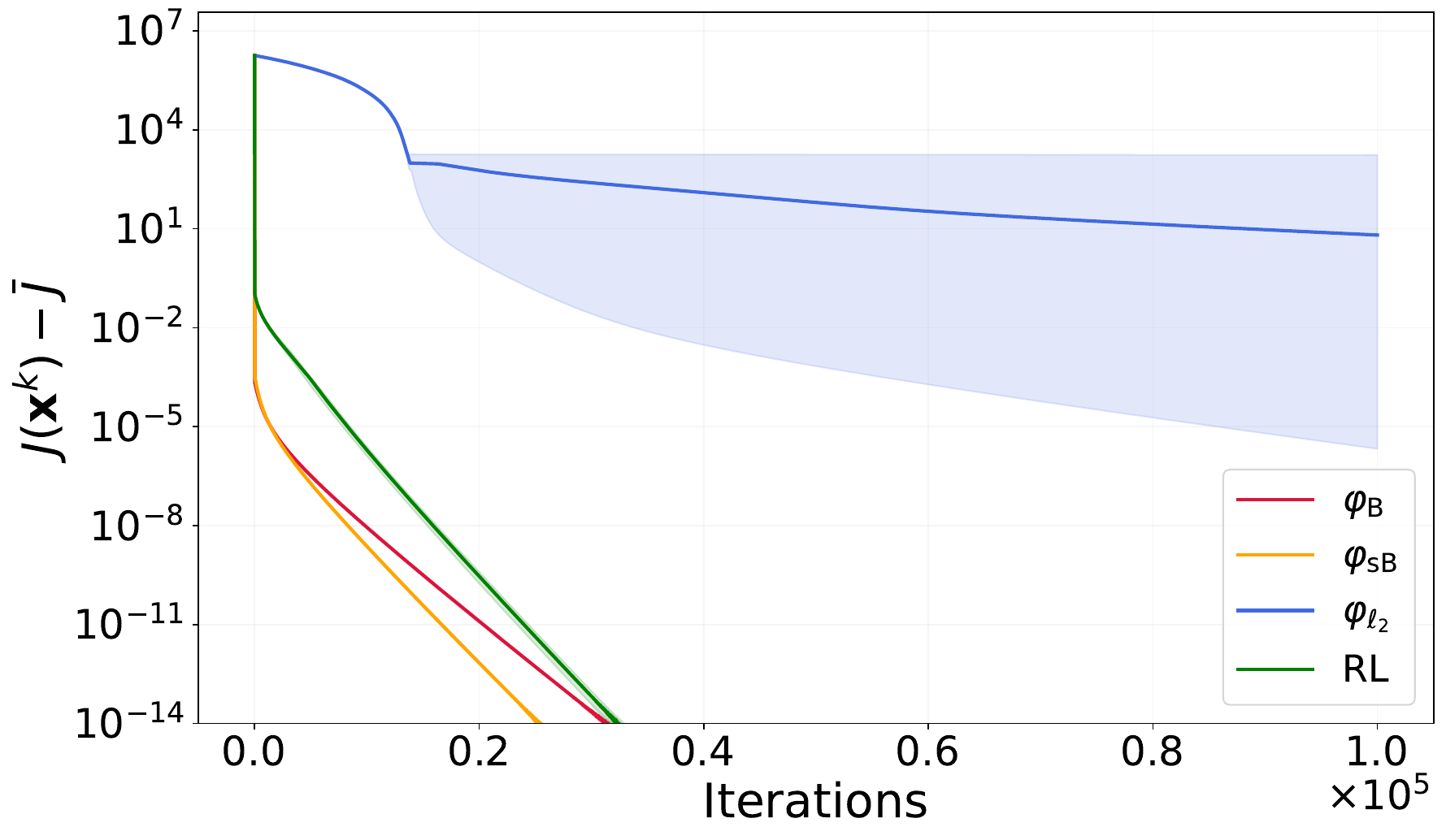} 
        \caption{$\bar\bx \in \mathrm{int}(\cC)$}
        \label{fig:fr_no_reg_int}
    \end{subfigure}
    \hfill 
    \begin{subfigure}[b]{0.48\textwidth}
        \centering
        \includegraphics[width=\textwidth]{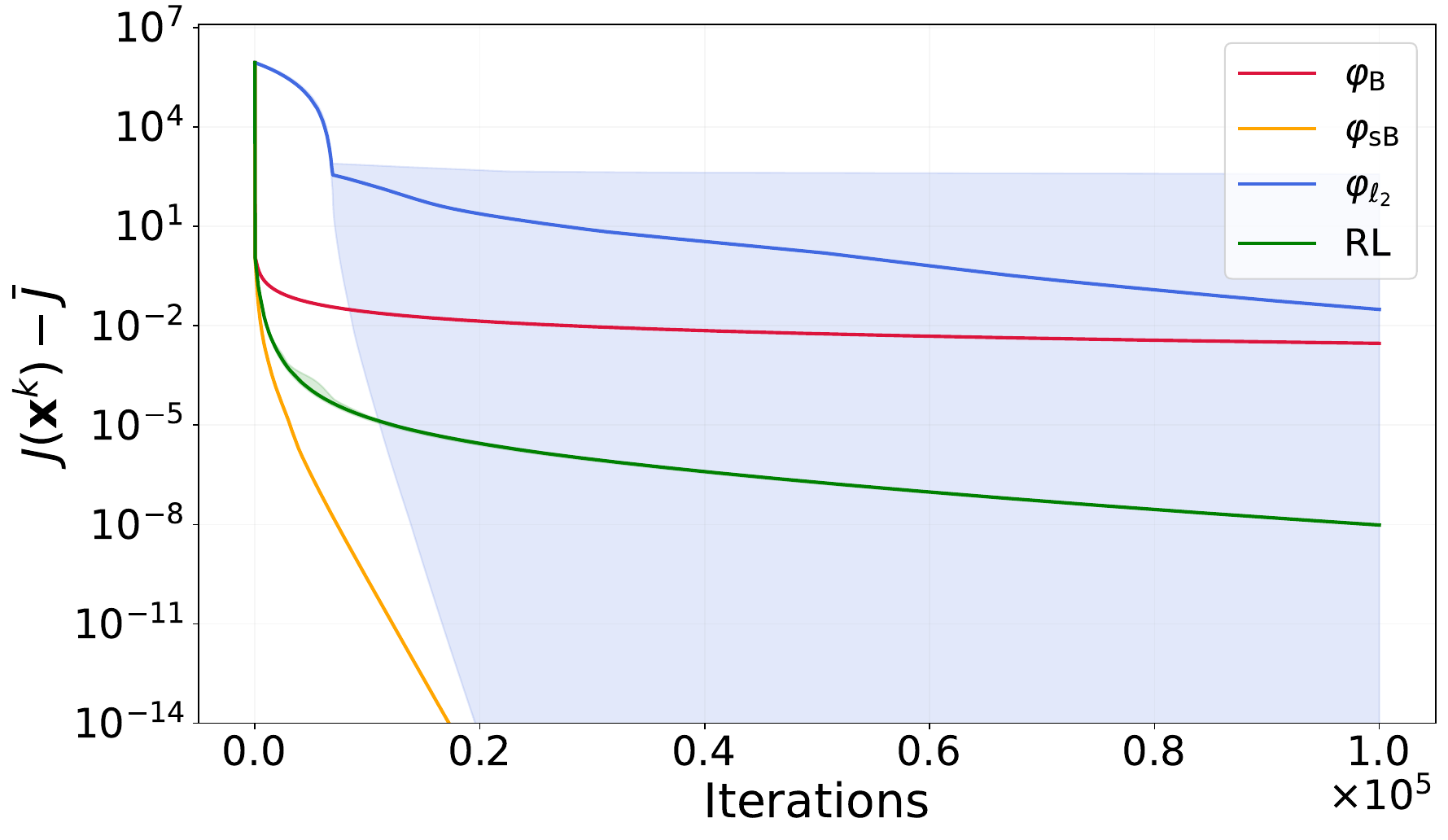} 
        \caption{$\bar\bx \in \partial \cC$}
        \label{fig:fr_no_reg_boundary}
    \end{subfigure}
    
    \caption{Convergence plots of~\cref{eq:BPGM}   with backtracking and the three choices of $\varphi$, as well as the RL algorithm, for unregularized KL in the case $N=M=500$. On the left ($\bar\bx\in\mathrm{int}(\cC)$), we observe linear convergence both for $\varphi_{\mathrm{B}}$ and $\varphi_{\mathrm{sB}}$ (and RL). On the right ($\bar\bx\in\partial\cC$), linear convergence is only observed for $\varphi_{\mathrm{sB}}$. In order to generate data such that $\bar{\bx} \in \operatorname{int}(\cC)$ we followed \Cref{app: generation sol in int C in unreg case}.}
    \label{fig:fr_no_reg}
\end{figure}

On the one hand,~\Cref{fig:nfr_no_reg,fig:fr_no_reg} present the results for unregularized KL problems, distinguishing between rank-deficient ($M<N$, infinity of solutions) and full rank ($M=N$, unique solution) regimes. As predicted by the theory, we systematically observe linear convergence for $\varphi_\mathrm{sB}$, regardless of whether the solution is unique or not. The situation differs for $\varphi_\mathrm{B}$: as expected from theoretical results, we observe linear convergence when the solution $\bar \bx$ (if unique) or the limit point $\bx^\infty$ (if the solution is non-unique) lies in $\mathrm{int}(\cC)$. Conversely, when $\bar \bx$ or $\bx^\infty$ lies on $\partial \cC$, we systematically observe a sublinear convergence behavior for $\varphi_\mathrm{B}$, thus confirming the sharpness of our theoretical analysis. Indeed, in these settings, we were not able to prove that \eqref{eq:rrsc} holds, and the observed empirical behavior confirms that this was not a gap in our theoretical results.

On the other hand,~\Cref{fig:NFR reg,fig:fr_reg} investigate the $\ell_1$-regularized setting. When $\bA$ does not have full rank, the KL divergence itself lacks a unique minimizer and, for these experiments, we only consider situations in which the composite objective $\mathrm{KL}(\by,\bA\cdot+\mathbf{b})+\lambda\|\cdot\|_1$ admits a  unique minimizer $\bar\bx$  (note that this is not always the case as the $\ell_1$-norm is not strictly convex). When $\cS_{\ker}^\theta = \emptyset$, the same observations as in the unregularized setting hold, and once again, they align perfectly with the theory. The novel situations here are those where $\cS_{\ker}^\theta \neq \emptyset$. In particular, as depicted in \Cref{fig:valid_regions}, the setting of~\Cref{fig:NFR reg} is such that $\cS_{\ker}^0=\Gamma^+_{\bA,0}(\bar\bx)\setminus\{\bar\bx\} \neq \emptyset$, being $\bar\bx$ the unique solution of our $\ell_1$-regularized KL problem. As already anticipated at the end of~\Cref{sec:KL_theory}, we cannot guarantee linear convergence only if iterates (or iterates after $K$ iterations) remain in $\Gamma^+_{\bA,0}(\bar\bx)$. This seems to be very unlikely since, as it can be seen in~\Cref{fig:NFR reg}, we systematically observed linear convergence in our experiments.

\begin{figure}[t]
    \centering
    \includegraphics[width=0.8\linewidth]{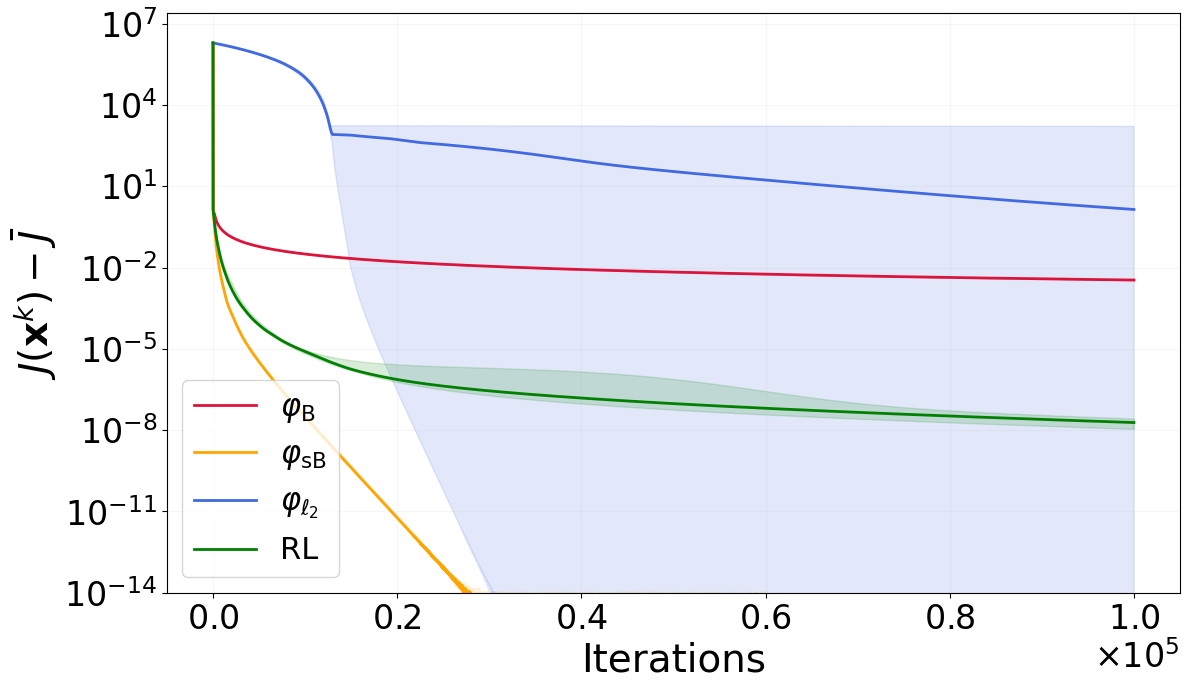}
    \caption{Convergence of~\cref{eq:BPGM}   with backtracking and the three choices of $\varphi$, as well as  the RL algorithm, for $\ell_1$-regularized KL in the case $N=500, \ M=400$ and $\bar\bx \in \partial \cC$. We only observe linear converge for $\varphi_{\mathrm{sB}}$. }
    \label{fig:NFR reg}
\end{figure}

Finally,   across all convergence plots, the standard proximal gradient update with backtracking (corresponding to~\cref{eq:BPGM}  with backtracking and $\varphi_{\ell_2}$) exhibits a strong local behavior, meaning that its convergence speed is highly sensitive to the initialization $\bx^0$ (see also~\Cref{fig:valid_regions}). This is again in line with our theoretical analysis.  

\begin{remark}\label{rem:NFR_reg_solIntC}
    Note that, for the $\ell_1$-regularized setting, we omitted experiments for the rank-deficient scenario with a unique $\bar\bx\in\mathrm{int}(\cC)$, as this is only possible under the very restrictive condition $M = N-1$. To see this, let $\bar \bx \in \mathrm{int}(\cC)$ be the solution of the $\ell_1$-regularized KL problem. Take a direction $\mathbf{d} \in (\ker(\bA) \cap \mathbf{1}^\perp)\setminus \{\mathbf{0}\}$ and define $\bz = \bar \bx + t \mathbf{d} \in \mathrm{int}(\cC)$ (for $|t|$ small  to remain in $\mathrm{int}(\cC)$). Then we have both $\bA \bz = \bA \bar \bx$ (as $\mathbf{d} \in \ker(\bA)$) and $\|\bz\|_1 = \left<\mathbf{1}, \bz\right> = \|\bar \bx\|_1 + t \left<\mathbf{1}, \mathbf{d}\right> =  \|\bar \bx\|_1$ (as $\mathbf{d} \in \mathbf{1}^\perp$). It follows that $\bz$ is also a minimizer of the problem and that $\bar \bx$ is not unique. For $\bar \bx$ being the unique minimizer we need to have $\ker(\bA) \cap \mathbf{1}^\perp = \{\mathbf{0}\}$ which, given $\dim(\mathbf{1}^\perp) = N-1$, implies $\dim(\ker(\bA)) \leq 1$ and thus $M\geq N-1$.
\end{remark}

\begin{figure}[t]
    \centering
    \begin{subfigure}[b]{0.48\textwidth}
        \centering
        \includegraphics[width=\textwidth]{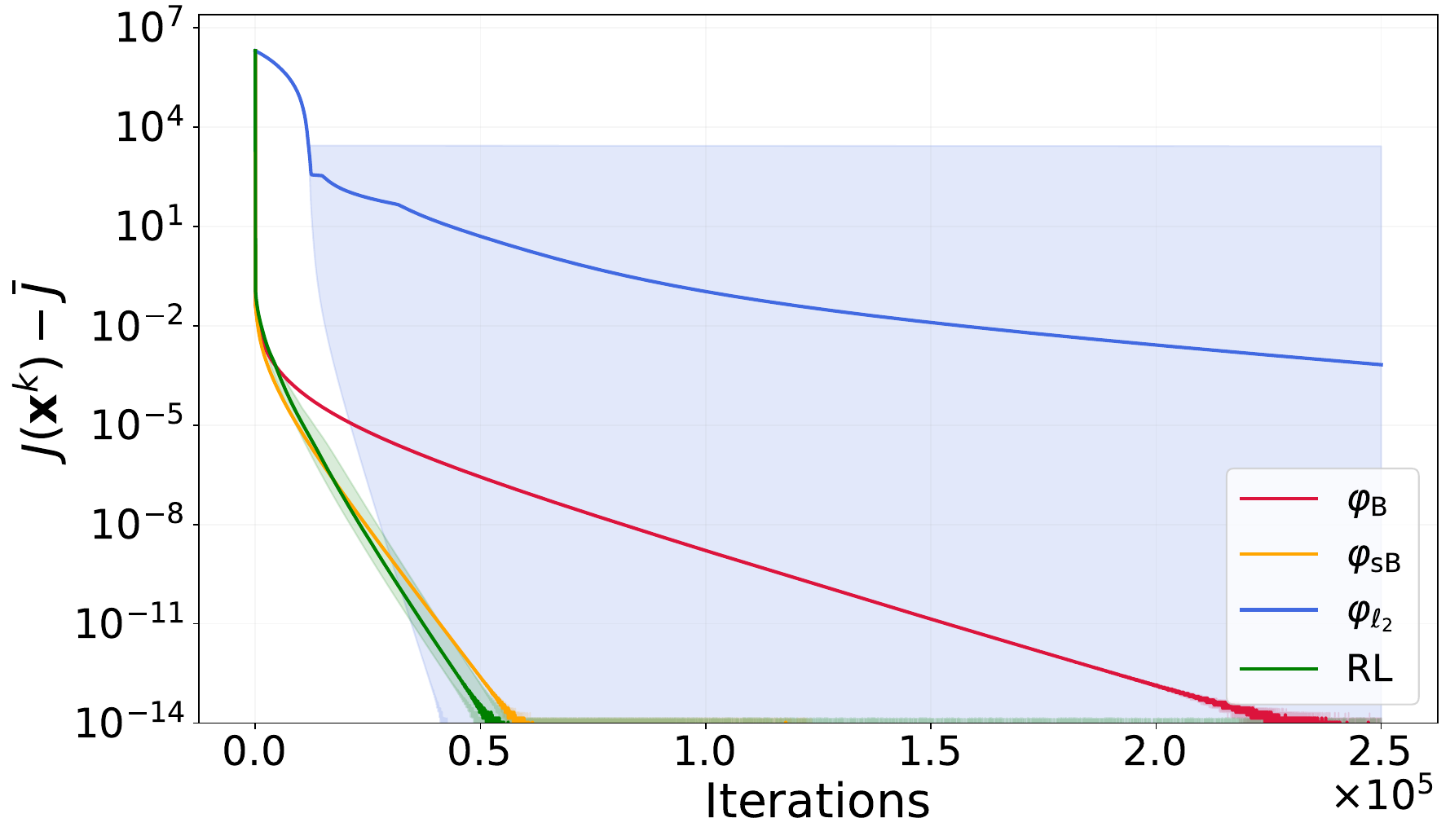} 
        \caption{$\bar\bx \in \mathrm{int}(\cC)$}
        \label{fig:fr_reg_int}
    \end{subfigure}
    \hfill 
    \begin{subfigure}[b]{0.48\textwidth}
        \centering
        \includegraphics[width=\textwidth]{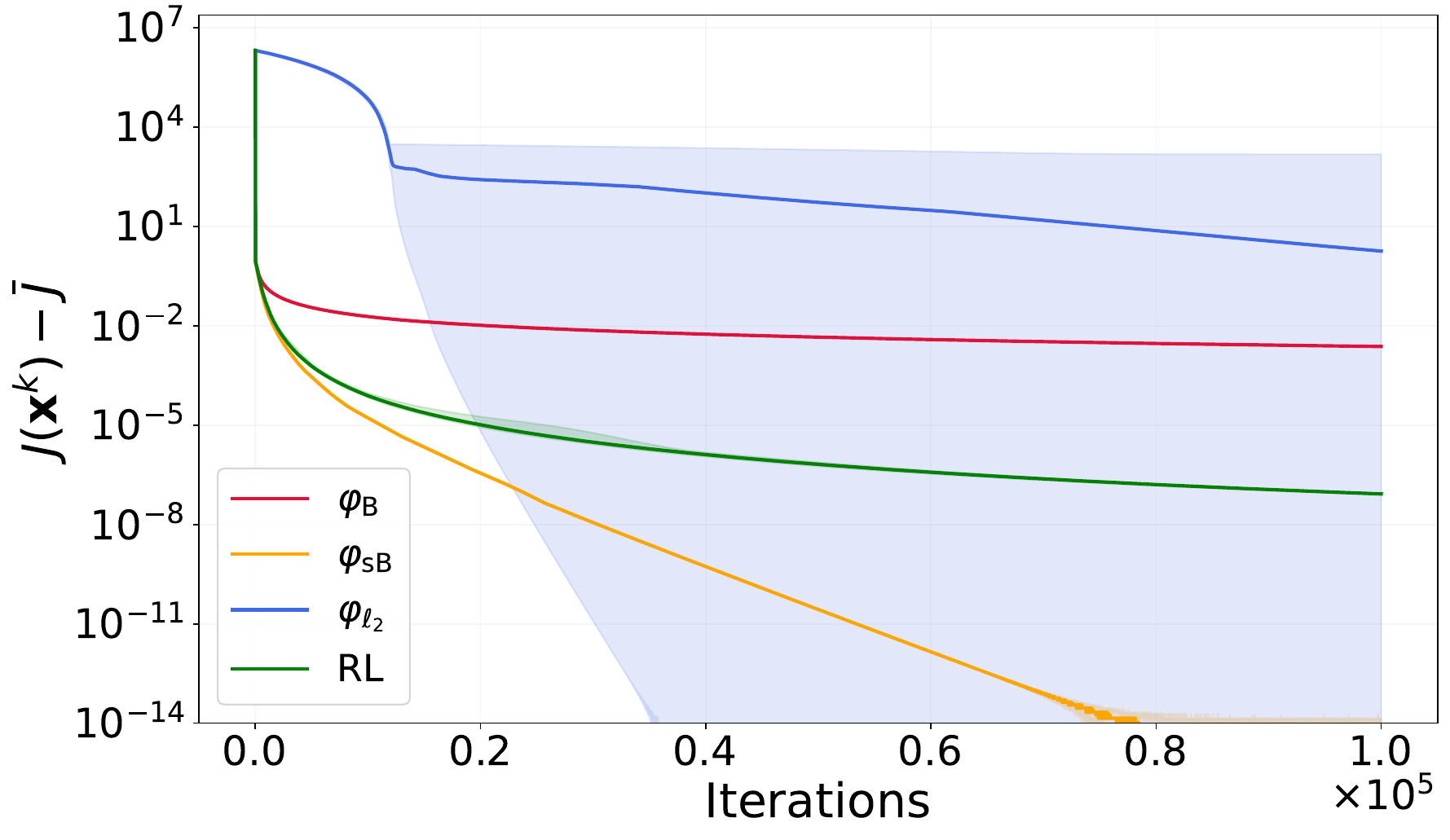} 
        \caption{$\bar\bx \in \partial \cC$}
        \label{fig:fr_reg_boundary}
    \end{subfigure}
    
    \caption{Convergence of~\cref{eq:BPGM}  with backtracking and the three choices of $\varphi$, as well as the RL algorithm, for $\ell_1$-regularized KL in the case $N=M=500$. On the left ($\bar\bx\in\mathrm{int}(\cC)$), we observe linear convergence both for $\varphi_{\mathrm{B}}$ and $\varphi_{\mathrm{sB}}$ (and RL). On the right ($\bar\bx\in\partial\cC$), linear convergence is only observed for $\varphi_{\mathrm{sB}}$. In order to generate data such that $\bar{\bx} \in \operatorname{int}(\cC)$ we followed \Cref{app: generation sol in int C in reg case}.}
    \label{fig:fr_reg}
\end{figure}

\subsubsection{Observations regarding practical performance} Our numerical experiments also provide a practical comparison of the evaluated algorithms. First, we point out that, across each figure, RL consistently beats (or has a comparable behavior in very restrictive cases such as that of~\Cref{fig:fr_no_reg_int})~\cref{eq:BPGM} with $\varphi_{\mathrm{B}}$. Such experimental remark has already been observed in other applications such as nonnegative matrix factorization~\cite{Hien2021}. This is not anymore the case with $\varphi_{\mathrm{sB}}$. Notably, \cref{eq:BPGM} with $\varphi_{\mathrm{sB}}$ outperformed RL in almost every scenario, either by exhibiting a faster linear convergence rate (see \Cref{fig:nfr_no_reg,fig:fr_no_reg_int}) or by achieving linear convergence where RL degraded to a sublinear rate (see \Cref{fig:fr_no_reg_boundary,fig:NFR reg,fig:fr_reg_boundary}). In the single instance where RL converged faster (\Cref{fig:fr_reg_int}), the overall performance remained comparable. We do note, however, a practical trade-off: while RL and~\cref{eq:BPGM} with $\varphi_{\mathrm{B}}$ are tuning-free, the $\varphi_{\mathrm{sB}}$ formulation requires selecting the smoothness parameter $\xi$. Finally, our varied initializations highlight the local nature of \cref{eq:BPGM} with $\varphi_{\ell_2}$ (standard PGM with backtracking). When initialized sufficiently close to the solution (i.e., the lowest curve of the shaded area), it performed quite well---achieving parity with the $\varphi_{\mathrm{sB}}$ variant in several cases (see \Cref{fig:fr_no_reg_boundary,fig:NFR reg,fig:fr_reg_int})--- and strictly outperforming all other algorithms in \Cref{fig:fr_reg_boundary}.

\section{Conclusions}
In this work, we established novel sufficient conditions to guarantee linear convergence of BPGM by introducing the Restricted Relative Strong Convexity (RRSC) property: by requiring relative strong convexity only between a point and its Euclidean projection onto the solution set, we significantly relax global assumptions previously required in the literature. 
We demonstrated the practical utility of this theoretical framework for KL regression problems. A key finding of our analysis highlights the limitations of the use of standard Burg's entropy as a distance generating function, which restricts linear convergence to cases where the solution set lies in the interior of the domain. To overcome this, we proposed to consider instead the smoothed Burg's entropy, and showed that it provides the appropriate geometric properties to satisfy the RRSC condition, ensuring linear convergence across a broader set of scenarios. Finally, our theoretical findings were thoroughly validated through numerical experiments
for both unregularized and $\ell_1$-regularized KL regression tasks. Moving forward, our theoretical framework opens several promising avenues for research. For instance, extending the applicability of our framework to other loss functions, such as general $\beta$-divergences, would significantly broaden its impact in inverse problems.

\section{Acknowledgements}
J.C.-R., C.F. and E.S. acknowledge support from the ANR EROSION project (ANR-22-CE48-0004) and the Toulouse AI cluster ANITI (ANR-23-IACL-0002). C. D. acknowledges support from the European Research Council (project MALIN, grant
101117133).
\bibliographystyle{plain}{}
\bibliography{sample}

@inproceedings{gopalan2015scalable,
  title={Scalable recommendation with hierarchical poisson factorization.},
  author={Gopalan, P. and Hofman, J. M. and Blei, D. M.},
  booktitle={Proc.~UAI},
  pages={326--335},
  year={2015}
}

@Article{lucy74,
  author    = {Lucy, L. B.},
  journal   = {Astronomical Journal},
  title     = {An iterative technique for the rectification of observed distributions},
  year      = {1974},
  pages     = {745--754},
  volume    = {79},
  doi       = {10.1086/111605},
}

@InProceedings{can04,
  author    = {Canny, J. F.},
  booktitle = {Proc.~ACM International Conference on Research and Development of Information Retrieval (SIGIR)},
  title     = {{G}a{P}: A Factor Model for Discrete Data},
  year      = {2004},
  pages     = {122--129},
}

@article{lu2018relatively,
  title={Relatively smooth convex optimization by first-order methods, and applications},
  author={Lu, H. and Freund, R. M. and Nesterov, Y.},
  journal={SIAM Journal on Optimization},
  volume={28},
  number={1},
  pages={333--354},
  year={2018},
  publisher={SIAM}
}

@article{lions1979,
  title={Splitting algorithms for the sum of two nonlinear operators},
  author={Lions, P.-L. and Mercier, B.},
  journal={SIAM Journal on Numerical Analysis},
  volume={16},
  number={6},
  pages={964--979},
  year={1979},
  publisher={SIAM}
}

@article{Bertero2009,
  doi = {10.1088/0266-5611/25/12/123006},
  year = {2009},
  month = {nov},
  publisher = {},
  volume = {25},
  number = {12},
  pages = {123006},
  author = {Bertero, M. and Boccacci, P. and Desiderà, G. and Vicidomini, G.},
  title = {Image deblurring with Poisson data: from cells to galaxies},
  journal = {Inverse Problems}
}

@book{Bertero2018,
  author = {Bertero, M. and Boccacci, P. and Ruggiero, V.},
  title = {Inverse Imaging with Poisson Data},
  publisher = {IOP Publishing},
  year = {2018},
  series = {2053-2563},
  isbn = {978-0-7503-1437-4},
  url_ = {https://doi.org/10.1088/2053-2563/aae109},
  doi = {10.1088/2053-2563/aae109}
}

@article{Bauschke2019,
  title={On linear convergence of non-{E}uclidean gradient methods without strong convexity and {L}ipschitz gradient continuity},
  author={Bauschke, H. H. and Bolte, J. and Chen, J. and Teboulle, M. and Wang, X.},
  journal={Journal of Optimization Theory and Applications},
  volume={182},
  number={3},
  pages={1068--1087},
  year={2019},
  publisher={Springer}
}

@inproceedings{Bartlett2007,
  author = {Bartlett, P. and Hazan, E. and Rakhlin, A.},
  booktitle = {Advances in Neural Information Processing Systems},
  editor_ = {J. Platt and D. Koller and Y. Singer and S. Roweis},
  pages = {},
  publisher_ = {Curran Associates, Inc.},
  title = {Adaptive Online Gradient Descent},
  volume = {20},
  year = {2007}
}

@inproceedings{liu2014,
  title={An asynchronous parallel stochastic coordinate descent algorithm},
  author={Liu, J. and Wright, S. and R{\'e}, C. and Bittorf, V. and Sridhar, S.},
  booktitle={International Conference on Machine Learning},
  pages={469--477},
  year={2014},
  organization={PMLR}
}

@ARTICLE{FI2011,
  author={Févotte, C. and Idier, J.},
  journal={Neural Computation}, 
  title={Algorithms for Nonnegative Matrix Factorization with the {$\beta$}-Divergence}, 
  year={2011},
  volume={23},
  number={9},
  pages={2421--2456},
  doi={10.1162/NECO_a_00168}
}

@article{Passty1979,
  title = {Ergodic convergence to a zero of the sum of monotone operators in Hilbert space},
  journal = {Journal of Mathematical Analysis and Applications},
  volume = {72},
  number = {2},
  pages = {383--390},
  year = {1979},
  issn = {0022-247X},
  doi = {10.1016/0022-247X(79)90234-8},
  url_ = {https://www.sciencedirect.com/science/article/pii/0022247X79902348},
  author = {Passty, G. B.}
}

@inproceedings{Gha2014,
  author = {Needell, D. and Srebro, N. and Ward, R.},
  booktitle = {Advances in Neural Information Processing Systems},
  pages = {},
  publisher_ = {Curran Associates, Inc.},
  title = {Stochastic Gradient Descent, Weighted Sampling, and the Randomized {K}aczmarz algorithm},
  url_ = {https://proceedings.neurips.cc/paper_files/paper/2014/file/b3310bba2be31e673a7ded3386994599-Paper.pdf},
  volume = {27},
  year = {2014}
}

@article{Zhou2019,
  title={A simple convergence analysis of {B}regman proximal gradient algorithm},
  author={Zhou, Y. and Liang, Y. and Shen, L.},
  journal={Computational Optimization and Applications},
  volume={73},
  number={3},
  pages={903--912},
  year={2019},
  publisher={Springer}
}

@article{Bredies_2008,
  title={Linear Convergence of Iterative Soft-Thresholding},
  volume={14},
  issn={1531-5851},
  url_={http://dx.doi.org/10.1007/s00041-008-9041-1},
  doi={10.1007/s00041-008-9041-1},
  number={5-6},
  journal={Journal of Fourier Analysis and Applications},
  publisher={Springer Science and Business Media LLC},
  author={Bredies, K. and Lorenz, D. A.},
  year={2008},
  month_=Oct, 
  pages={813--837} 
}

@article{ChenTeboulle1993,
  author = {Chen, G. and Teboulle, M.},
  title = {Convergence Analysis of a Proximal-Like Minimization Algorithm Using {B}regman Functions},
  journal = {SIAM Journal on Optimization},
  volume = {3},
  number = {3},
  pages = {538--543},
  year = {1993},
  doi = {10.1137/0803026}
}

@article{chirinos2025,
  title={Optimization Landscape of {$\ell_0$}-{B}regman Relaxations},
  author={Chirinos-Rodríguez, J. and Févotte, C. and Soubies, E.},
  journal={ArXiv Preprint},
  year={2025},
}

@ARTICLE{FessHe1995,
  author={Fessler, J. A. and Hero, A. O.},
  journal={IEEE Transactions on Image Processing}, 
  title={Penalized maximum-likelihood image reconstruction using space-alternating generalized EM algorithms}, 
  year={1995},
  volume={4},
  number={10},
  pages={1417--1429},
  doi={10.1109/83.465106}
}

@article{Bolte2007,
  author = {Bolte, J. and Daniilidis, A. and Lewis, A.},
  year = {2007},
  pages = {1205--1223},
  title = {The {{\L}}ojasiewicz Inequality for Nonsmooth Subanalytic Functions with Applications to Subgradient Dynamical Systems},
  volume = {17},
  journal = {SIAM Journal on Optimization},
  doi = {10.1137/050644641}
}

@article{Kurdyka1998,
  author = {Kurdyka, K.},
  title = {On gradients of functions definable in o-minimal structures},
  journal = {Annales de l'Institut Fourier},
  pages = {769--783},
  year = {1998},
  publisher = {Association des Annales de l{\textquoteright}institut Fourier},
  volume = {48},
  number = {3},
  doi = {10.5802/aif.1638}
}

@article{Beck2009,
  author = {Beck, A. and Teboulle, M.},
  title = {A Fast Iterative Shrinkage-Thresholding Algorithm for Linear Inverse Problems},
  journal = {SIAM Journal on Imaging Sciences},
  volume = {2},
  number = {1},
  pages = {183--202},
  year = {2009},
  doi = {10.1137/080716542}
}

@article{CW2005,
  author = {Combettes, P. L. and Wajs, V. R.},
  title = {Signal Recovery by Proximal Forward-Backward Splitting},
  journal = {Multiscale Modeling \& Simulation},
  volume = {4},
  number = {4},
  pages = {1168--1200},
  year = {2005},
  doi = {10.1137/050626090}
}

@article{loj1963,
  title={Une propri{\'e}t{\'e} topologique des sous-ensembles analytiques r{\'e}els, in “Les {\'E}quations aux D{\'e}riv{\'e}es Partielles”}, 
  journal={{\'E}ditions du Centre National de la Recherche Scientifique},
  author={{\L}ojasiewicz, S.},
  year={1963},
  publisher={Paris}
}

@article{necoara2019linear,
  title={Linear convergence of first order methods for non-strongly convex optimization},
  author={Necoara, I. and Nesterov, Y. and Glineur, F.},
  journal={Mathematical programming},
  volume={175},
  number={1},
  pages={69--107},
  year={2019},
  publisher={Springer}
}

@article{Drus2018,
  author = {Drusvyatskiy, D. and Lewis, A. S.},
  title = {Error Bounds, Quadratic Growth, and Linear Convergence of Proximal Methods},
  journal = {Mathematics of Operations Research},
  volume = {43},
  number = {3},
  pages = {919--948},
  year = {2018},
  doi = {10.1287/moor.2017.0889}
}

@article{zhang2015,
  title={Restricted strong convexity and its applications to convergence analysis of gradient-type methods in convex optimization},
  author={Zhang, H. and Cheng, L.},
  journal={Optimization Letters},
  volume={9},
  number={5},
  pages={961--979},
  year={2015},
  publisher={Springer}
}

@article{beck2017,
  title={Linearly convergent away-step conditional gradient for non-strongly convex functions},
  author={Beck, A. and Shtern, S.},
  journal={Mathematical Programming},
  volume={164},
  number={1},
  pages={1--27},
  year={2017},
  publisher={Springer}
}

@article{Neco2016,
  title={Parallel random coordinate descent method for composite minimization: Convergence analysis and error bounds},
  author={Necoara, I. and Clipici, D.},
  journal={SIAM Journal on Optimization},
  volume={26},
  number={1},
  pages={197--226},
  year={2016},
  publisher={SIAM}
}

@article{Essafri2024,
  title={Exact continuous relaxations of $\ell_0$-regularized criteria with non-quadratic data terms},
  author={Essafri, M. and Calatroni, L. and Soubies, E.},
  journal={Journal of Global Optimization},
  volume={93},
  number={3},
  pages={651--699},
  year={2025},
  publisher={Springer}
}

@INPROCEEDINGS{EssafriKL,
  author={Essafri, M. and Calatroni, L. and Soubies, E.},
  booktitle={2024 IEEE 34th International Workshop on Machine Learning for Signal Processing (MLSP)}, 
  title={On $\ell_{0}$ {B}regman-Relaxations for {K}ullback--{L}eibler Sparse Regression}, 
  year={2024},
  volume={},
  number={},
  pages={1--6},
  doi={10.1109/MLSP58920.2024.10734806}
}

@article{Zhou2017,
  title={A unified approach to error bounds for structured convex optimization problems},
  author={Zhou, Z. and So, A. M.-C.},
  journal={Mathematical Programming},
  volume={165},
  number={2},
  pages={689--728},
  year={2017},
  publisher={Springer}
}

@inproceedings{Nesterov1983,
  title={A method for solving the convex programming problem with convergence rate O (1/k2)},
  author={Nesterov, Y.},
  booktitle={Dokl akad nauk Sssr},
  volume={269},
  pages={543},
  year={1983}
}

@book{BCombettes,
  author = {Bauschke, H. H. and Combettes, P. L.},
  title = {Convex analysis and monotone operator theory in {H}ilbert spaces},
  series = {CMS Books in Mathematics/Ouvrages de Math\'{e}matiques de la SMC},
  edition = {Second},
  publisher = {Springer, Cham},
  year = {2017},
  pages = {xix+619},
  isbn = {978-3-319-48310-8; 978-3-319-48311-5},
  mrclass = {49-02 (41A65 46B20 46C05 47H05 90C25)},
  mrnumber = {3616647},
  doi = {10.1007/978-3-319-48311-5},
  url_ = {https://doi.org/10.1007/978-3-319-48311-5},
}

@article{Donoho2005,
  author = {Donoho, D. L. and Tanner, J.},
  title = {Sparse nonnegative solution of underdetermined linear equations by linear programming},
  journal = {Proceedings of the National Academy of Sciences},
  volume = {102},
  number = {27},
  pages = {9446--9451},
  year = {2005},
  doi = {10.1073/pnas.0502269102}
}

@article{Boyer2019,
  author = {Boyer, C. and Chambolle, A. and Castro, Y. D. and Duval, V. and de Gournay, F. and Weiss, P.},
  title = {On Representer Theorems and Convex Regularization},
  journal = {SIAM Journal on Optimization},
  volume = {29},
  number = {2},
  pages = {1260--1281},
  year = {2019},
  doi = {10.1137/18M1200750}
}

@article{Lai2013,
  title={Augmented $\ell_1$ and Nuclear-Norm Models with a Globally Linearly Convergent Algorithm},
  volume={6},
  issn={1936-4954},
  doi={10.1137/120863290},
  number={2},
  journal={SIAM Journal on Imaging Sciences},
  publisher={Society for Industrial & Applied Mathematics (SIAM)},
  author={Lai, M.-J. and Yin, W.},
  year={2013},
  pages={1059--1091} 
}

@article{Peypo2017,
  title={From error bounds to the complexity of first-order descent methods for convex functions},
  author={Bolte, J. and Nguyen, T. P. and Peypouquet, J. and Suter, B. W.},
  journal={Mathematical Programming},
  volume={165},
  number={2},
  pages={471--507},
  year={2017},
  publisher={Springer}
}

@article{Li2017,
  title={Calculus of the Exponent of {K}urdyka--{{\L}}ojasiewicz Inequality and Its Applications to Linear Convergence of First-Order Methods},
  volume={18},
  issn={1615-3383},
  doi={10.1007/s10208-017-9366-8},
  number={5},
  journal={Foundations of Computational Mathematics},
  publisher={Springer Science and Business Media LLC},
  author={Li, G. and Pong, T. K.},
  year={2017},
  pages={1199--1232} 
}

@article{harmany2011spiral,
  title={This is SPIRAL-TAP: Sparse Poisson intensity reconstruction algorithms—theory and practice},
  author={Harmany, Z. T. and Marcia, R. F. and Willett, R. M.},
  journal={IEEE Transactions on Image Processing},
  volume={21},
  number={3},
  pages={1084--1096},
  year={2011},
  publisher={IEEE}
}

@article{Rock1970,
  title={Convex Analysis},
  author={Rockafellar, R. T.},
  journal={Princeton University Press},
  volume={28},
  year={1970}
}

@article{Attouch2013,
  title={Convergence of descent methods for semi-algebraic and tame problems: proximal algorithms, forward--backward splitting, and regularized Gauss--Seidel methods},
  author={Attouch, H. and Bolte, J. and Svaiter, B. F.},
  journal={Mathematical programming},
  volume={137},
  number={1},
  pages={91--129},
  year={2013},
  publisher={Springer}
}

@article{Teboulle2018,
  author = {Teboulle, M.},
  year = {2018},
  pages = {},
  title = {A simplified view of first order methods for optimization},
  volume = {170},
  journal = {Mathematical Programming. Ser. B},
  doi = {10.1007/s10107-018-1284-2}
}

@inproceedings{karimi2016,
  title={Linear convergence of gradient and proximal-gradient methods under the {P}olyak-{{\L}}ojasiewicz condition},
  author={Karimi, H. and Nutini, J. and Schmidt, M.},
  booktitle={Joint European conference on machine learning and knowledge discovery in databases},
  pages={795--811},
  year={2016},
  organization={Springer}
}

@book{Nesterov2013,
  title = {Introductory {Lectures} on {Convex} {Optimization}: {A} {Basic} {Course}},
  isbn = {978-1-4419-8853-9},
  shorttitle = {Introductory {Lectures} on {Convex} {Optimization}},
  publisher = {Kluwer Academic Publishers, Boston},
  author = {Nesterov, Y.},
  year = {2003}
}

@article{Bolte2018,
  author = {Bolte, J. and Sabach, S. and Teboulle, M. and Vaisbourd, Y.},
  title = {First Order Methods Beyond Convexity and {L}ipschitz Gradient Continuity with Applications to Quadratic Inverse Problems},
  journal = {SIAM Journal on Optimization},
  volume = {28},
  number = {3},
  pages = {2131--2151},
  year = {2018},
  doi = {10.1137/17M1138558}
}

@book{HUL2001,
  location = {Berlin, Heidelberg},
  title = {Fundamentals of Convex Analysis},
  rights = {http://www.springer.com/tdm},
  isbn = {978-3-540-42205-1 978-3-642-56468-0},
  url = {http://link.springer.com/10.1007/978-3-642-56468-0},
  publisher = {Springer},
  author = {Hiriart-Urruty, J.-B. and Lemaréchal, C.},
  year = {2001},
  doi = {10.1007/978-3-642-56468-0}
}

@inproceedings{LS2000,
  author = {Lee, D. and Seung, H. S.},
  booktitle = {Advances in Neural Information Processing Systems},
  editor = {T. Leen and T. Dietterich and V. Tresp},
  pages = {},
  publisher_ = {MIT Press},
  title = {Algorithms for Non-negative Matrix Factorization},
  url_ = {https://proceedings.neurips.cc/paper_files/paper/2000/file/f9d1152547c0bde01830b7e8bd60024c-Paper.pdf},
  volume = {13},
  year = {2000}
}

@ARTICLE{Fevotte2009,
  author={Févotte, C. and Bertin, N. and Durrieu, J.-L.},
  journal={Neural Computation}, 
  title={Nonnegative Matrix Factorization with the {I}takura-{S}aito Divergence: With Application to Music Analysis}, 
  year={2009},
  volume={21},
  number={3},
  pages={793--830},
  doi={10.1162/neco.2008.04-08-771}
}

@article{bauschke2017descent,
  title={A descent lemma beyond {L}ipschitz gradient continuity: first-order methods revisited and applications},
  author={Bauschke, H. H. and Bolte, J. and Teboulle, M.},
  journal={Mathematics of Operations Research},
  volume={42},
  number={2},
  pages={330--348},
  year={2017},
  publisher={Informs}
}

@article{Bauschke1997,
  author = {Bauschke, H. G. and Borwein, J. M.},
  journal = {Journal of Convex Analysis},
  language = {eng},
  number = {1},
  pages = {27--67},
  publisher = {Heldermann Verlag},
  title = {{L}egendre functions and the method of random {B}regman projections.},
  url = {http://eudml.org/doc/227096},
  volume = {4},
  year = {1997}
}

@article{Luo1993,
  title={Error bounds and convergence analysis of feasible descent methods: a general approach},
  author={Luo, Z.-Q. and Tseng, P.},
  journal={Annals of Operations Research},
  volume={46},
  number={1},
  pages={157--178},
  year={1993},
  publisher={Springer}
}

@article{adly2026,
  title={Variable {B}regman Majorization-Minimization algorithms for nonconvex nonsmooth optimization, with application to Poisson imaging}, 
  author={Adly, M. and Chazottes, A. and Chouzenoux, E. and Pesquet, J.-C. and Sureau, F.},
  journal={ArXiv preprint},
  year={2026}
}

@article{AusTeb2006,
  author = {Auslender, A. and Teboulle, M.},
  title = {Interior Gradient and Proximal Methods for Convex and Conic Optimization},
  journal = {SIAM Journal on Optimization},
  volume = {16},
  number = {3},
  pages = {697--725},
  year = {2006},
  doi = {10.1137/S1052623403427823}
}

@article{Tseng2010,
  title={Approximation accuracy, gradient methods, and error bound for structured convex optimization},
  author={Tseng, P.},
  journal={Mathematical Programming},
  volume={125},
  number={2},
  pages={263--295},
  year={2010},
  publisher={Springer}
}

@phdthesis{dey20043d,
  title={3D microscopy deconvolution using {R}ichardson--{L}ucy algorithm with total variation regularization},
  author={Dey, N. and Blanc-F{\'e}raud, L. and Zimmer, C. and Roux, P. and Kam, Z. and Olivo-Marin, J.-C. and Zerubia, J.},
  year={2004},
  school={INRIA}
}

@article{Zunino_2023,
  title={Reconstructing the Image Scanning Microscopy Dataset: An Inverse Problem},
  author={Zunino, A. and Castello, M. and Vicidomini, G.},
  journal={Inverse Problems},
  volume={39},
  number={6},
  pages={064004},
  year={2023},
  publisher={IOP Publishing}
}

@article{Hien2021,
   title={Algorithms for Nonnegative Matrix Factorization with the Kullback–Leibler Divergence},
   volume={87},
   ISSN={1573-7691},
   url_={http://dx.doi.org/10.1007/s10915-021-01504-0},
   DOI={10.1007/s10915-021-01504-0},
   number={3},
   journal={Journal of Scientific Computing},
   publisher={Springer Science and Business Media LLC},
   author={Hien, Le Thi Khanh and Gillis, Nicolas},
   year={2021}}

@article{daubechies2004iterative,
  title={An iterative thresholding algorithm for linear inverse problems with a sparsity constraint},
  author={Daubechies, I. and Defrise, M. and De Mol, C.},
  journal={Communications on Pure and Applied Mathematics: A Journal Issued by the Courant Institute of Mathematical Sciences},
  volume={57},
  number={11},
  pages={1413--1457},
  year={2004},
  publisher={Wiley Online Library}
}

@article{richardson1972bayesian,
  author  = {Richardson, W. H.},
  title   = {Bayesian-based iterative method of image restoration},
  journal = {Journal of the Optical Society of America},
  volume  = {62},
  number  = {1},
  pages   = {55--59},
  year    = {1972},
  doi     = {10.1364/JOSA.62.000055}
}

\appendix

\crefalias{section}{appendix}
\crefalias{subsection}{appendix} 
\crefalias{subsubsection}{appendix} 

\section{Proofs of~\Cref{sec:KL}}\label{appendix}

\subsection{Proof of~\Cref{prop:rrsc for KL}}\label{app:rrsc} 
To prove the result, we first need the following technical lemma.
\begin{lemma}\label{lem:anglecond}
    Let $\bA\in\R^{M\times N}$ and $\theta\in(0,\pi/2)$. Then, for any two points  $\bx, \bz\in\R^N$ such that $\angle (\bx-\bz,\ker(\bA))> \theta$ we have 
    \begin{equation}
        \|\bA(\bx-\bz)\|_2> \sigma_S\sin(\theta) \|\bx-\bz\|_2.
    \end{equation}
where $\sigma_S>0$ is the smallest non-zero singular value of $\bA$.
\end{lemma}

\begin{proof} Let $\bx, \bz\in\R^N$ such that $\angle (\bx-\bz,\ker(\bA))> \theta$. Then, there exists $\mathbf{w}\in\ker(\bA)^\perp$ such that $\bA(\bx-\bz)=\bA\mathbf{w}$. Additionally, let $\mathrm{rank}(\bA)=S\leq\min\{M,N\}$ and consider the singular value decomposition of $\bA$, $\bA=\mathbf{U}\boldsymbol{\Sigma}\mathbf{V}^T$, where $\mathbf{U}\in\R^{M\times M}$ and $\mathbf{V}\in\R^{N\times N}$ are orthogonal matrices and $\boldsymbol{\Sigma}=\mathrm{diag}(\sigma_1,\ldots,\sigma_S,0,\ldots,0)\in\R^{M\times N}$, being $\sigma_s>0$, $s=1,\ldots, S$, the singular values of $\bA$. With this, there exist coefficients $c_s>0$, $s=1,\ldots, S$ such that $\bw=\sum_{s=1}^S c_s\mathbf{v}_s$(the first $S$ columns of $\mathbf{V}$ form a basis of $\ker(\bA)^\perp$), which implies that $\|\bw\|_2^2=\sum_{s=1}^s c_s^2$. We have
$$
\bA\bw=\bA\left(\sum_{s=1}^S c_s\mathbf{v}_s\right)=\sum_{s=1}^S c_s\bA\mathbf{v}_s=\sum_{s=1}^S c_s\sigma_s\mathbf{u}_s,
$$
and so
$$
\|\bA\bw\|_2^2=\|\sum_{s=1}^S c_s\sigma_s\mathbf{u}_s\|_2^2=\sum_{s=1}^Sc_s^2\sigma_s^2\langle\mathbf{u}_s,\mathbf{u}_s\rangle=\sum_{s=1}^Sc_s^2\sigma_s^2.
$$
The above leads to
$$
\|\bA\bw\|_2^2\geq\sigma_S^2\sum_{s=1}^Sc_s^2=\sigma_{S}^2\|\bw\|_2^2.
$$
Now, since $\angle (\bx-\bz,\ker(\bA))> \theta$ by assumption, we have that $\|\bw\|_2>\sin(\theta)\|\bx-\bz\|_2$. Combining everything, we have that
$$
\|\bA(\bx-\bz)\|_2=\|\bA\bw\|_2\geq\sigma_S\|\bw\|_2>\sigma_S\sin(\theta)\|\bx-\bz\|_2.
$$
This concludes the proof.
\end{proof}
We are now ready to prove the desired result.

\subsubsection{\texorpdfstring{Proof for $\varphi_{\mathrm{sB}}$}{Smoothed Burgs entropy varphisb}}\label{proof: brsc}

Let $\theta\in(0,\pi/2)$. We aim at showing that there exists $\mu>0$ such that
\begin{equation}\label{eq:RRSC_for_KL_1}
\sum_{m =1}^{M} \frac{y_{m}[\bA(\bx-\bar\bx_{\ell_2})]_m^2}{([\bA\bx]_m+b_m)([\bA\bar\bx_{\ell_2}]_m+b_m)}
\geq 
\mu \sum_{n=1}^N \frac{(x_n-\bar x_{\ell_2,n})^2}{(x_n + \xi)(\bar x_{\ell_2,n} + \xi )}
\end{equation}
for all $\bx \in \cC\setminus\cS_{\ker}^\theta$.
We now aim at simplifying the above estimate by finding an appropriate lower bound for the left hand side. To do so, let us define the following constant
$$
\delta_1:= \frac{\min_{m}y_m}{\max_{\bz\in \cS} \|\bA\mathbf{z}+\mathbf{b}\|_2}>0.
$$
Observe that $\delta_1$ is strictly positive since $\by>0$ by assumption and the maximum in the denominator is attained because $\mathcal{S}$ is compact (\Cref{lem:sol set KL}). We then have, using  $0 \leq [\bA\bx]_m \leq \|\bA\bx\|_2 \leq \|\bA\|_2 \|\bx\|_2$, that
\begin{equation*}\label{eq:rrsc before lemma}
\begin{aligned}
\sum_{m =1}^{M} \frac{y_{m}[\bA(\bx-\bar\bx_{\ell_2})]_m^2}{([\bA\bx]_m+b_m)([ \bA\bar\bx_{\ell_2}]_m+b_m)}\geq
\delta_1\sum_{m =1}^{M} \frac{[\bA(\bx-\bar\bx_{\ell_2})]_{m}^2}{\|\bA\|_2\|\bx\|_2+b_m}\geq\delta_1\frac{\|\bA(\bx-\bar\bx_{\ell_2})\|_2^2}{\|\bA\|_2\|\bx\|_2+\max_mb_m}
\end{aligned}
\end{equation*}
for all $\bx\in\cC\setminus\cS_{\ker}^\theta$.  Now, since $\bx\in\cC\setminus\cS_{\ker}^\theta$, we have that $\angle(\bx-\bar\bx_{\ell_2},\ker(\bA))> \theta$ and hence, by~\Cref{lem:anglecond} we get that $\|\bA(\bx-\bar\bx_{\ell_2})\|_2> \sigma_S\sin(\theta)\|\bx-\bar\bx_{\ell_2}\|_2$, which leads to
\begin{equation}\label{eq:rrsc_proof_final_estimate}
\begin{aligned}
\sum_{m =1}^{M} \frac{y_{m}[\bA(\bx-\bar\bx_{\ell_2})]_{m}^2}{([\bA\bx]_{m}+ b_m)([\bA\bar\bx_{\ell_2}]_{m}+b_m)}
&\geq
\delta_1\frac{\|\bA(\bx-\bar\bx_{\ell_2})\|_2^2}{\|\bA\|_2\|\bx\|_2+\max_m b_m}\\
&> 
\delta_1\sigma_S\sin(\theta)\frac{\|\bx-\bar\bx_{\ell_2}\|_2^2}{\|\bA\|_2\|\bx\|_2+\max_m b_m}.
\end{aligned}
\end{equation}
The rest of the proof follows exactly the same as the one of~\cite[Theorem 2.7]{chirinos2025}. For this reason, we omit it and only recall here the behaviour of $\mu$ as a function of $\xi$: we have that
\begin{equation}\label{eq:mu_estimate}
    \mu \geq \min \left\lbrace c_1 \xi^2, \frac{c_2\xi}{c_3 + \xi} \right\rbrace,
\end{equation}
for some $c_1,c_2,c_3>0$. 


\subsubsection{\texorpdfstring{Proof for $\varphi_{\mathrm{B}}$}{Burg's entropy varphiB}}
Let $\theta\in(0,\pi/2)$ and $\varepsilon>0$. In order to prove the result, we will show that, within $\cK_\varepsilon$, the left hand side of~\cref{eq:rrsc} for $\varphi_{\mathrm{B}}$ can be written as that of $\varphi_{\mathrm{sB}}$. Let $\bx\in\cK_\varepsilon$. Then, we have that $\bx, \bar \bx_{\ell_2} \in \R_{\geq \varepsilon}^N$ and, thus, there exist, $\mathbf{v}_1,\mathbf{v_2}\in\cC$ such that $\bx=\mathbf{v}_1+\mathbf{1}\eps$,  $\bar \bx_{\ell_2} = \mathbf{v}_2 + \mathbf{1}\varepsilon$. We then have that
$$
\begin{aligned}
\langle \nabla \varphi_\mathrm{B}(\bx) -\nabla \varphi_\mathrm{B}(\bar\bx_{\ell_2}),\bx-\bar\bx_{\ell_2} \rangle=
\sum_{n=1}^N\frac{(x_{n}-\bar x_{\ell_2,n})^2}{x_n \bar x_{\ell_2,n}}=
\sum_{n=1}^N\frac{(v_{1,n}-\bar v_{2,n})^2}{(v_{1,n}+\varepsilon)(\bar v_{2,n}+\varepsilon)}
\end{aligned}
$$
where the last term equals the left hand side of~\cref{eq:rrsc} when $\varphi=\mathrm{sBurg}$ for $\mathbf{v}\in\cC$, as desired.

\subsubsection{Proof for $\varphi_{\ell_2}$}
When $\bA$ has full column rank, $F$ is locally strongly convex since one can locally lower bound the smallest eigenvalue of its Hessian. When $\bA$ does not have full column rank, let $\mathcal{B} \subset \cC$ be any bounded set containing $\cS$. Then we get from this boundedness that
$$
\delta_2 := \min_{\bx \in \mathcal{B}} \frac{1}{\|\bA\|_2 \|\bx\|_2 + \max_m b_m} >0.
$$
Combining this with~\eqref{eq:rrsc_proof_final_estimate}, we get that, for all $\bx \in \mathcal{C} \setminus \cS^{\theta}_{\ker}$
$$
\sum_{m =1}^{M} \frac{y_{m}[\bA(\bx-\bar\bx_{\ell_2})]_m^2}{([\bA\bx]_m+b_m)([ \bA\bar\bx_{\ell_2}]_m+b_m)}\geq
\delta_1c\frac{\|\bx-\bar\bx_{\ell_2}\|_2^2}{\|\bA\|_2\|\bx\|_2+\max_m b_m} \geq \delta_1 \delta_2 c \|\bx-\bar\bx_{\ell_2}\|_2^2
$$
for some constants $\delta_1, \ c>0$ defined previously. This completes the proof.

\subsection{Proof of~\Cref{prop:restr_alpha_KL}}\label{sec:alphaKL_is_positive}

The proof of the result will be based on the following Lemma, which enlightens the fact that if $\varphi$ is separable, then it suffices to prove the result in one dimension.

\begin{lemma}\label{lem:from_n_to_one}
Let $\varphi\in\Gamma_0(\R^N)$ be a separable Legendre function defined by $\varphi(\bx):=\sum_{n=1}^N\phi_n(x_n)$ for all $\bx$, where each $\phi_n\in\Gamma_0(\R)$ is a Legendre function for $n=1,\ldots, N$. Consider sets $\cD_n\subseteq \mathrm{int}(\dom \phi_n)$ and their associated restricted symmetry coefficients $\alpha_{\cD_n}(\phi_n)\in[0,1]$. For any $\cD\subseteq \cD_1\times\cdots\times\cD_N$, the (restricted) symmetry coefficient of $\varphi$, $\alpha_\cD(\varphi)$, satisfies $\alpha_{\cD}(\varphi) \geq \min_{n=1,...,N} \alpha_{\cD_n}(\phi_n)$.
\end{lemma}

\begin{proof}
First, we notice that, since $\varphi$ is separable, then so is its associated Bregman divergence: for all $\bx, \by\in\mathrm{int}(\dom \varphi)$ we have that $D_{\varphi}(\bx,\by)=\sum_{n=1}^N d_{\phi_n}(x_n, y_n)$, where $d_{\phi_n}(\cdot,\cdot)$ denotes the one-dimensional Bregman divergence associated to $\phi_n$. Let $\cD\subseteq\cD_1\times\cdots\times\cD_N$. Then, by the definition of $\alpha_{\cD_n}(\phi_n)$, $n=1,\ldots, N$, it holds that
$$
D_{\varphi}(\bx,\by)=\sum_{n=1}^N d_{\phi_n}(x_n,y_n) \geq \min_{n=1,...,N} \alpha_{\cD_n}(\phi_n) \sum_{i=1}^N d_{\phi_n}(y_n,x_n)= \min_{n=1,...,N} \alpha_{\cD_n}(\phi_n) D_{\varphi}(\by,\bx),
$$
for all $\bx\in\mathrm{int}(\dom \varphi), \ \by\in\cD$ with $\bx\neq \by$. Notice that, as $\cD\subseteq \cD_1\times\cdots\times \cD_N$, we have that $y_m\in\cD_m$ for all $m=1,\ldots, M$. The above expression directly implies that  
$$
\frac{D_{\varphi}(\bx,\by)}{D_{\varphi}(\by,\bx)} \geq \min_{n=1,...,N} \alpha_{\cD_n}(\phi_n)
$$
for all $\bx\in\mathrm{int}(\dom \varphi), \ \by\in\cD$ with $\bx\neq\by$. Taking the infimum over $\bx$ and $\by$, we obtain the definition of $\alpha_\cD(\varphi)$, which yields
$$
\alpha_{\mathcal{D}}(\varphi) \geq \min_{n=1,...,N} \alpha_{\cD_n}(\phi_n).
$$
This completes the proof.
\end{proof}

\subsubsection{\texorpdfstring{Proof for $\varphi_\mathrm{sB}$}{varphi=sBurg}}

As already indicated, the above result allows us to prove the claim only for the one-dimensional setting $\phi_\xi(x):= - \log (x+\xi)$.

\begin{lemma} \label{thm: alpha(h)>0}
Let $\phi_\xi:=- \log(x+\xi)$, $\xi>0$, be the one-dimensional smoothed Burg's entropy. Then, for any compact set $\mathcal{S}\subset \mathbb{R}_{\geq 0}$, it holds that
\begin{equation}
\alpha_{\mathcal{S}}(\phi_\xi):=\inf \left\{\frac{d_{\phi_\xi} (x,y)}{d_{\phi_\xi}(y,x)} \mid x\in \mathbb{R}_{\geq 0}, \ y \in \mathcal{S}, x\neq y\right\} >0.
\end{equation}
\end{lemma}
\begin{proof}
Note that the definition of $\alpha_{\mathcal{S}}(\varphi)$ can be rewritten as
$$
\alpha_{\mathcal{S}}(\varphi)=\inf_{x\in\R_{\geq 0}} m(x), \quad \text{where } m(x)=\inf_{y\in \cS} \Delta_x(y),
$$
where, for all $x\in\R_{\geq 0}$,
$$
\Delta_x\colon \R\to \R_{\geq 0}; \quad \Delta_x(y):=
\begin{cases}
\frac{d_{\phi_\xi}(x, y)}{d_{\phi_\xi}(y, x)} & \text{ if } x\neq y\\
1, & \text{ if } x=y.
\end{cases}
$$
To prove the result, we start by showing that $m$ is well defined and continuous. To do so, we first claim that, for all $x\in \R_{\geq 0}$, $\Delta_x$ is continuous. First, if $x\neq y$, then it is continuous as the division of two Bregman divergences, the denominator being nonzero everywhere. It is therefore sufficient to prove that 
$$
\lim_{x\to y}\frac{d_{\phi_\xi}(x, y)}{d_{\phi_\xi}(y, x)}=1.
$$
We have that
$$
\lim_{x\to y}\frac{d_{\phi_\xi}(x, y)}{d_{\phi_\xi}(y, x)}=\lim_{x\to y}\frac{\log(\frac{y+\xi}{x+\xi})+\frac{x-y}{y+\xi}}{ \log(\frac{x+\xi}{y+\xi})+\frac{y-x}{x+\xi}}=\frac{\log(t)+1/t-1}{-\log(t)+t-1},
$$
where simply $t:=(y+\xi)/(x+\xi)$. In this case, as $x\to y$ is equivalent to $t\to 1$, we need to prove that
$$
\lim_{t\to 1}\frac{\log(t)+1/t-1}{-\log(t)+t-1}=1.
$$
By applying L'H\^opital's rule twice, we derive that
$$
\begin{aligned}
\lim_{t\to 1} \frac{\log(t)+1/t-1}{-\log(t)+t-1}=\lim_{t\to 1}\frac{t\log(t)+1-t}{-t\log(t)+t^2-t}= \lim_{t\to 1} \frac{\log(t)}{-\log(t)+2t-2}=\lim_{t\to 1} \frac{1/t}{-1/t+2}=1,
\end{aligned}
$$
concluding the proof of the claim. Now, as $\Delta_x(y)>0$ for all $y\in\cS$, $\cS$ is compact, and $\Delta_x$ is continuous, we get by the extreme value theorem that the minimum over $\mathcal{S}$ is attained, and so $m(x)=\min_{y\in \cS} \Delta_x(y)>0$ for all $x\in \R$. Additionally, as $m$ is defined as the minimum of continuous functions over a compact set, it is continuous too. Thus to conclude it suffices to show that $m$ is coercive. We want to show that 
$$
    \lim_{x\to\infty} \ \underset{y \in \cS}{\inf}\, \frac{d_{\phi_\xi}(x,y)}{d_{\phi_\xi}(y,x)} = \underset{x \to + \infty}{\lim} \ \underset{y \in \cS}{\inf} \, \frac{\log(\frac{y+\xi}{x+\xi})+\frac{x-y}{y+\xi}}{ \log(\frac{x+\xi}{y+\xi})+\frac{y-x}{x+\xi}}=+\infty.
$$
To do so, we will use compactness of $\mathcal{S}$ to lower bound (resp. upper bound) the numerator (resp. denominator) of the ratio above in such a way that it does not depend on $y$ anymore. We start observing that, since $\mathcal{S}$ is compact, there exist $0<\delta_1\leq\delta_2<+\infty$ such that $y\in[\delta_1,\delta_2]$. It follows that $\log(y+\xi)\geq \log(\delta_1)$ and
$$
\frac{x-y}{y+\xi}\geq \frac{x-y}{\delta_2+\xi}\geq\frac{x-\delta_1}{\delta_2+\xi}.
$$
Then, the numerator satisfies,
$$
\begin{aligned}
d_{\phi_\xi}(x,y)\geq  -\log(x+\xi)+\log(\delta_1)+\frac{x-\delta_1}{\delta_2+\xi}.
\end{aligned}
$$
On the other hand, we have that $\frac{y-x}{x+\xi}\leq \frac{y}{x+\xi}\leq \frac{\delta_2}{\xi}$, and so, the denominator satisfies
$$
 d_{\phi_\xi}(y,x)\leq -\log(\xi)+\log(x+\xi)+\frac{\delta_2}{\xi}.
$$
Combining both estimates, we have that
\begin{equation}\label{eq:lower_bound_m}
m(x)=\underset{y \in \cS}{\inf} \, \frac{d_{\phi_\xi}(x,y)}{d_{\phi_\xi}(y,x)}\geq\frac{-\log(x+\xi)+\log(\delta_1)+\frac{x-\delta_1}{\delta_2+\xi}}{-\log(\xi)+\log(x+\xi)+\frac{\delta_2}{\xi}},
\end{equation}
where the right hand side above clearly satisfies that
$$
\underset{x \to + \infty}{\operatorname{lim}} \frac{-\log(x+\xi)+\log(\delta_1)+\frac{x-\delta_1}{\delta_2+\xi}}{-\log(\xi)+\log(x+\xi)+\frac{\delta_2}{\xi}}= + \infty.
$$
Taking limits on both sides of \eqref{eq:lower_bound_m} leads to the desired result. 
\end{proof}
The desired result is a direct consequence of combining both ~\Cref{lem:from_n_to_one} and~\Cref{thm: alpha(h)>0}. Specifically, in~\Cref{thm: alpha(h)>0} we consider, for all $n=1,\ldots, N$,
$$
\cD_n:=\pi_n(\cS)=\{u \in\R_{\geq 0} \ \mid \ \text{there exists } \bar \bx \in \cS \text{ s.t. } \bar x_n = u\};
$$
i.e., for all $n$ we define $\cD_n$ as the Euclidean projection of $\cS$ onto the $n^{\mathrm{th}}$ component. Then,~\Cref{lem:from_n_to_one} holds for $\cS$ since by definition $\cS\subseteq\cD_1\times\ldots\times\cD_N$.

We add a final remark emphasizing the importance on the compactness of $\mathcal{S}$.

\begin{remark}
    In the proof of~\Cref{thm: alpha(h)>0}, the compactness of $\mathcal{S}$, in particular its boundedness was crucial. It is possible to show that if $\mathcal{S}$ is not bounded then $\alpha_{\mathcal{S}}(\phi_\xi)=0$. It suffices to notice that, for all $x \in \mathbb{R}_{\geq 0}$:
    $$
\underset{y \to + \infty}{\operatorname{lim}} \frac{d_{\phi_\xi}(x,y)}{d_{\phi_\xi}(y,x)}=\underset{y \to + \infty}{\operatorname{lim}} \frac{\log (y+\xi)+\log(x+\xi)+\frac{x-y}{y+\xi}}{\log(x+\xi)-\log(y+ \xi)+ \frac{y-x}{x+\xi}}= 0.
    $$
In particular we get that 
$$
0 \leq \alpha_{\mathcal{S}}(\phi_\xi) = \underset{x \in \mathbb{R}_{\geq 0}}{\inf} m(x) \leq  m(1) = \underset{y \in S}{\inf} \Delta_1 (y) \leq \underset{y \to + \infty}{\operatorname{lim}} \frac{d_{\phi_\xi}(1,y)}{d_{\phi_\xi}(y,1)}=0.
$$
\end{remark}

\subsubsection{Proof for $\varphi_\mathrm{B}$}
As done for $\varphi_\mathrm{sB}$, thanks to \Cref{lem:from_n_to_one} it is sufficient to prove the claim  for the one-dimensional setting. As we only want to prove the result over $\R_{\geq \varepsilon}$ we get
\begin{align*}
    \alpha_{\mathcal{S} \cap \R_{\geq \varepsilon}}(\varphi_\mathrm{B})&:=  \inf \left\{\frac{d_{\varphi_\mathrm{B}} (x,y)}{d_{\varphi_\mathrm{B}}(y,x)} \mid x\in \mathbb{R}_{\geq \varepsilon}, \ y \in \mathcal{S} \cap \R_{\geq \varepsilon}, x\neq y\right\} \\
    & =  \inf \left\{\frac{d_{\varphi_\mathrm{B} }(x+\varepsilon,y+\varepsilon)}{d_{\varphi_\mathrm{B}}(y + \varepsilon,x+\varepsilon)} \mid x\in \mathbb{R}_{\geq 0}, \ y \in (\mathcal{S} - \varepsilon) \cap \R_{\geq 0}, x\neq y\right\} \\
    & =  \inf \left\{\frac{d_{\varphi_\mathrm{B}(\cdot + \varepsilon)} (x,y)}{d_{\varphi_\mathrm{B}(\cdot+\varepsilon)}(y ,x)} \mid x\in \mathbb{R}_{\geq 0}, \ y \in (\mathcal{S} - \varepsilon) \cap \R_{\geq 0}, x\neq y\right\} >0
\end{align*}
where the $>0$ is a direct application of the case $\varphi_{\mathrm{sB}}$ proved in the previous section, using the fact that $(\mathcal{S} - \varepsilon) \cap \R_{\geq 0}$ is bounded as $\mathcal{S}$ is.

\subsection{Proof of~\Cref{prop:rel_smooth_KL}}\label{proof:rel_smooth_KL}

\subsubsection{\texorpdfstring{Proof for $\varphi_\mathrm{sB}$}{varphi=sBurg}}
As they will be used in the following, we provide here the expressions of the gradient and Hessian of $\varphi_\mathrm{sB}$,
\begin{equation}
    \nabla \varphi_\mathrm{sB}(\bx) = - \frac{1}{\bx + \bm{\xi} }, \; \text{ and, } \; \nabla^2\varphi_\mathrm{sB}(\bx) = \frac{1}{(\bx + \bm{\xi})^2},
\end{equation}
where simply $\boldsymbol{\xi}=(\xi,\ldots, \xi)$ and division and square operations have to be understood component-wise. Given that both $F$ and $\varphi_\mathrm{sB}$ are twice differentiable, the relative smoothness property described in~\cref{eq:rel_Lip} is equivalent to 
\begin{equation}
L \left<\nabla^2 \varphi_\mathrm{sB}(\bx)\mathbf{d},\mathbf{d}  \right>  \geq \left<\nabla^2 \mathrm{KL}(\by,\bA\bx+\mathbf{b})\mathbf{d},\mathbf{d} \right>, \quad \text{for all } \bx \in \R_{\geq 0}^N, \, \mathbf{d} \in \R^N.
\end{equation}
First, observe that
\begin{align}
& \left<\nabla^2 \varphi_\mathrm{sB}(\bx)\mathbf{d},\mathbf{d}  \right> = \sum_{n=1}^N \frac{d_n^2}{(x_n + \xi)^2}, \label{eq:hess_cond_h} \\ 
    & \left<\nabla^2 \mathrm{KL}(\by,\bA\bx+\mathbf{b})\mathbf{d},\mathbf{d} \right> = \sum_{m=1}^M \frac{y_m}{([\bA \bx]_m +b_m)^2} [\bA \mathbf{d}]_m^2. \label{eq:hess_cond_F}
\end{align}
Now, notice that, by the Cauchy-Schwartz inequality,
$$
\begin{aligned}
[\bA \mathbf{d}]_m^2&=\left(\sum_{n=1}^Na_{m,n}d_n\right)^2=\left(\sum_{n=1}^N\sqrt{a_{m,n}(x_n+\xi)}\frac{\sqrt{a_{m,n}}d_n}{\sqrt{x_n+\xi}}\right)^2\\
&\leq\left(\sum_{n=1}^N a_{m,n}(x_n+\xi)\right)\left(\sum_{n=1}^N\frac{a_{m,n}d_n^2}{x_n+\xi}\right)=[\bA(\bx+\boldsymbol{\xi})]_m\left(\sum_{n=1}^N\frac{a_{m,n}d_n^2}{x_n+\xi}\right).
\end{aligned}
$$
Plugging this into~\cref{eq:hess_cond_F} leads to
$$
\begin{aligned}
\left<\nabla^2 \mathrm{KL}(\by,\bA\bx+\mathbf{b})\mathbf{d},\mathbf{d} \right>&= \sum_{m=1}^M \frac{y_m}{([\bA \bx]_m +b_m)^2} [\bA \mathbf{d}]_m^2\\
&\leq 
\sum_{m=1}^M\frac{y_m[\bA(\bx+\boldsymbol{\xi})]_m}{([\bA \bx]_m +b_m)^2}\sum_{n=1}^N\frac{a_{m,n}d_n^2}{x_n+\xi}\\
&=\sum_{n=1}^N\frac{d_n^2}{x_n+\xi}\sum_{m=1}^M\frac{y_ma_{m,n}[\bA(\bx+\boldsymbol{\xi})]_m}{([\bA \bx]_m +b_m)^2}\\
&=\sum_{n=1}^N\frac{d_n^2}{(x_n+\xi)^2}\sum_{m=1}^M\frac{y_ma_{m,n}(x_n+\xi)[\bA(\bx+\boldsymbol{\xi})]_m}{([\bA \bx]_m +b_m)^2}.
\end{aligned}
$$
By defining $R_m(\bx):=[\bA(\bx+\boldsymbol{\xi})]_m/([\bA \bx]_m +b_m)^2$, $m=1,\ldots, M$, the above can be condensed as
\begin{equation}\label{eq:estimate_F_RLS}
    \left<\nabla^2 \mathrm{KL}(\by,\bA\bx+\mathbf{b})\mathbf{d},\mathbf{d} \right>\leq \sum_{n=1}^N\frac{d_n^2}{(x_n+\xi)^2}\sum_{m=1}^My_ma_{m,n}(x_n+\xi)R_m(\bx).
\end{equation}
Next, for every $m=1,\ldots, M$ consider the function $g_m(t):=(t+[\bA\boldsymbol{\xi}]_m)/(t+b_m)$, $t\geq 0$. We will now find an upper bound on $g_m$ that does not depend on $t$ by computing $\max_{t\geq 0} g_m(t)$. Observe that, if $[\bA\boldsymbol{\xi}]_m\leq b_m$, then $g(z)\leq 1$. On the other hand, if $[\bA\boldsymbol{\xi}]_m> b_m$, then notice that $g_m'(t)=(b_m-[\bA\boldsymbol{\xi}]_m)/(t+b_m)^2<0$, which implies that $g_m$ is decreasing for all $t\geq 0 $ and so $g_m(t)\leq g_m(0)=[\bA\boldsymbol{\xi}]_m/b_m$. Thus, $g_m(t)\leq \max\{1,[\bA\boldsymbol{\xi}]_m/b_m\}=:\omega_m$ for all $m=1,\ldots, M$ and all $t\geq 0$. Consequently, we have that
$$
R_m(\bx)=\frac{g_m([\bA\bx]_m)^2}{[\bA(\bx+\boldsymbol{\xi})]_m}\leq \frac{\omega_m^2}{[\bA(\bx+\boldsymbol{\xi})]_m}
$$
for all $m=1,\ldots, M$. Finally, plugging the above estimate into~\cref{eq:estimate_F_RLS} leads to
$$
\begin{aligned}
\left<\nabla^2 \mathrm{KL}(\by,\bA\bx+\mathbf{b})\mathbf{d},\mathbf{d} \right>
&\leq 
\sum_{n=1}^N\frac{d_n^2}{(x_n+\xi)^2}\sum_{m=1}^M\omega_m^2y_m\frac{a_{m,n}(x_n+\xi)}{[\bA(\bx+\boldsymbol{\xi})]_m}\\
&\leq 
\left(\sum_{m=1}^M\omega_m^2y_m\right)\sum_{n=1}^N\frac{d_n^2}{(x_n+\xi)^2},
\end{aligned}
$$
where in the last inequality we used that $\frac{a_{m,n}(x_n+\xi)}{[\bA(\bx+\boldsymbol{\xi})]_m}\leq 1$ for all $m=1,\ldots,M$, concluding. 

\subsubsection{\texorpdfstring{Proof for $\varphi_\mathrm{B}$}{varphi=Burg}}
For clarity, we point out that the proof of the present result was given in~\cite{bauschke2017descent}[Lemma 7], but for $\mathrm{KL}(\by,\bA\bx)$. By simply noticing that
$$
\begin{aligned}
\langle \nabla^2\mathrm{KL}(\by,\bA\bx+\mathbf{b})\mathbf{d},\mathbf{d} \rangle&=\sum_{m=1}^M \frac{y_m}{([\bA \bx]_m+b_m)^2} [\bA \mathbf{d}]_m^2\\
&\leq \sum_{m=1}^M \frac{y_m}{[\bA \bx]_m^2} [\bA \mathbf{d}]_m^2\\
&= \langle \nabla^2\mathrm{KL}(\by,\bA\bx)\mathbf{d},\mathbf{d} \rangle
\end{aligned}
$$
we get the desired result.

\section{Further details on data generation}

\subsection{Details on the construction of the matrix $\bA$} \label{app: construction of A}

The matrix $\bA$ of Section \ref{sec: data gen} is constructed in the following way:

\begin{enumerate}
    \item We generate $\tilde{\bA}$ sampling its entries $\tilde{a}_{m,n} \sim \cU(0,1)$ and normalizing the resulting matrix $\bar{\bA}$ as $\tilde{\bA} / \|\tilde{\bA}\|_2$.
    \item We compute the SVD decomposition of $\bar{\bA} = \mathbf{U} \bar{\bm{\Sigma}} \mathbf{V}^{T}$. We consider a target condition number $\kappa$  and we rescale $\{\bar{\sigma_i}\}_{i=1}^M$ the singular values of $\bar{\bA}$ with the following formula:
$$
        \sigma_{i} = \bar{\sigma}_{\max} - \left(\bar{\sigma}_{\max} - \frac{\bar{\sigma}_{\max}}{\kappa}\right) \frac{\bar{\sigma}_{\max}- \bar{\sigma}_i}{\bar{\sigma}_{\max}- \bar{\sigma}_{\min}}
$$
    We consider $\bA = \mathbf{U} \bm{\Sigma} \mathbf{V}^{T}$, with $\bm{\Sigma} = \operatorname{diag}(\sigma_1,\dots,\sigma_M)$.
    \item We take $\bA = |\bA|$, i.e. we consider the matrix obtained taking the absolute value of all the entries of $\bA$. 
\end{enumerate}
 Under this generation process, it holds that $\operatorname{rank}(\bA) = \min(M,N)$ with probability 1.\\

Notice that, after the second step above, we actually have that $\operatorname{cond}(\bA) = \kappa$, while after applying the absolute value in step 3, this is not guaranteed anymore. Nevertheless the final condition number obtained is much smaller than the starting one as showed in \Cref{tab:cond_number_evolution}

\begin{table}[h]
\setlength{\extrarowheight}{0.1cm}
\centering
\caption{Evolution of the condition number of the matrixes obtained in the steps to generate $\bA$ from $\bar{\bA}$ with $M=N=500$ and $\kappa=50$.}
\label{tab:cond_number_evolution}
\begin{tabular}{ccc}
\hline
 $\bar{\bA}$ (Step 1) &  $\bA$ (Step 2) & {$|\bA|$} (Step 3) \\ \hline
18531.47                  & 50.00                       & 181.77                                \\ \hline
\end{tabular}
\end{table}

\subsection{\texorpdfstring{Details on the generation of the solution $\bar{\bx} \in \operatorname{int}(\cC)$}{Details on the generation of the solution x }} \label{app: construction of bar x in int C}

The procedure presented in \Cref{sec: data gen} to construct $\bx^*$ and $\by$ when $\bA$ has full rank leads with high probability (see \Cref{rem:implicit_reg}) to a solution $\bar{\bx} \in \partial \cC$. In order to exhibit the behavior of \eqref{eq:BPGM} when the solution $\bar{\bx} \in \mathrm{int}(\cC)$, we need to generate the data in a specific way, which we detail below.

\subsubsection{Unregularized case}
\label{app: generation sol in int C in unreg case}

For the unregularized case, we simply consider $\mathbf{x}^* = \mathbf{1} $ and $\by = \bA \mathbf{x}^* + \mathbf{b}$, i.e., removing the Poisson noise. In this case we have that $\bar{\bx} = \bx^* \in \operatorname{int}(\cC)$.

\subsubsection{Regularized case} 
\label{app: generation sol in int C in reg case}

To enforce the optimal solution $\bar{\mathbf{x}}$ to lie strictly in $\operatorname{int}(\cC)$, we exploit the Karush-Kuhn-Tucker (KKT) optimality conditions.
Consider the regularized minimization problem:
\begin{equation}\label{eq:reg_problem_app}
\bx^* \in \mathrm{arg}\, \min_{\mathbf{x} \ge \mathbf{0}} \; \operatorname{KL}(\mathbf{y}, \mathbf{A}\mathbf{x} + \mathbf{b}) + \lambda \|\mathbf{x}\|_1
\end{equation}
If the optimal solution $\mathbf{x}^*$ is strictly positive, the non-negativity constraints are not active, and the subgradient of the $\ell_1$-norm reduces to the vector of ones $\mathbf{1}$. Therefore, the first-order optimality condition implies that the gradient of \eqref{eq:reg_problem_app} must vanish at $\mathbf{x}^*$:
\begin{equation}\label{eq:kkt_interior}
\mathbf{A}^T \left( \mathbf{1} - \frac{\mathbf{y}}{\mathbf{A}\mathbf{x}^* + \mathbf{b}} \right) + \lambda \mathbf{1} = \mathbf{0}
\end{equation}
where the division between vectors is applied component-wise. Now, the idea is to construct $\by$ such that $\bx^* \in \operatorname{int}(\cC)$ is the unique solution of the problem in \eqref{eq:reg_problem_app}. To do so, we proceed as follows.
\begin{enumerate}
    \item We generate the target ground-truth vector $\mathbf{x}^* \in \operatorname{int}(\cC)$, e.g. by sampling from a uniform distribution.
    \item We compute $\mathbf{w}$ as the solution of the following linear system:
$$
    \mathbf{A}^T \mathbf{w} = \mathbf{A}^T \mathbf{1} + \lambda \mathbf{1}.
$$ 
Then, from \eqref{eq:kkt_interior} we can identify $\mathbf{w}$ as: $w_m = \frac{y_m}{[\mathbf{A}\mathbf{x}^* + \mathbf{b}]_m}$ for all $m =1,\ldots,M$.
    \item We construct  $\mathbf{y}   = (\mathbf{A}\mathbf{x}^* + \mathbf{b}) \odot \mathbf{w}$ (component-wise multiplication).
\end{enumerate}

\end{document}